\documentclass{amsart}
\usepackage{amssymb,color}
\usepackage{hyperref}
\usepackage{graphicx,url}

\newtheorem{theorem}{Theorem}[section]
\newtheorem{proposition}[theorem]{Proposition}
\newtheorem{lemma}[theorem]{Lemma}
\newtheorem{corollary}[theorem]{Corollary}

\theoremstyle{definition}
\newtheorem{definition}[theorem]{Definition}
\newtheorem{example}[theorem]{Example}

\theoremstyle{remark}

\numberwithin{equation}{section}

\DeclareMathOperator\rank{rank}

\newcommand{\sym}{{\sf Sym}}
\newcommand{\C}{\mathbb{C}}
\newcommand{\Q}{\mathbb{Q}}

\newcommand{\cA}{\mathcal{A}}
\newcommand{\cB}{\mathcal{B}}
\newcommand{\cC}{\mathcal{C}}
\newcommand{\cD}{\mathcal{D}}
\newcommand{\cF}{\mathcal{F}}
\newcommand{\cH}{\mathcal{H}}
\newcommand{\cI}{\mathcal{I}}
\newcommand{\cJ}{\mathcal{J}}
\newcommand{\cL}{\mathcal{L}}
\newcommand{\cM}{\mathcal{M}}
\newcommand{\cP}{\mathcal{P}}
\newcommand{\cR}{\mathcal{R}}

\newcommand{\cS}{\mathcal{S}}

\newcommand{\mtrx}[4]{\begin{bmatrix} #1 & #2\\ #3 & #4 \end{bmatrix}}

\newcommand{\pr}{{\sf pr}}
\newcommand{\ovr}[1]{\overline{#1}}
\newcommand{\lm}{\lambda}
\newcommand{\veps}{\varepsilon}
\newcommand{\fuse}[2]{#1/#2}
\newcommand{\aaut}{{\sf AAut}}

\begin{document}

\title[Uniformity in association schemes]{Uniformity in association
schemes and coherent configurations: cometric Q-antipodal schemes and linked systems}

\author[Edwin R. Van Dam]{Edwin R. van Dam}
\address{Department of  Econometrics and Operations Research,
Tilburg University, PO Box 90153, 5000 LE Tilburg, The Netherlands}
\email{Edwin.vanDam@uvt.nl}

\author{William J. Martin}
\address{Department of Mathematical Sciences and Department of Computer Science,
Wor\-ces\-ter Polytechnic Institute, 100 Institute Rd, Worcester, MA 01609,
USA} \email{martin@wpi.edu}
\thanks{The second author was supported in part by NSA grant number H98230-07-1-0025.\\ \indent This version is published in
Journal of Combinatorial Theory, Series A 120 (2013), 1401--1439.}

\author{Mikhail Muzychuk}
\address{Department of Mathematics, Netanya Academic College,
University St. 1, Netanya 42365, Israel} \email{muzy@netanya.ac.il}

\subjclass[2010]{Primary 05E30, Secondary 05B25, 05C50, 51E12}


\dedicatory{Dedicated to the memory of Donald G. Higman}

\keywords{cometric association scheme, imprimitivity,
Q-antipodal association scheme, uniform association scheme,
linked system, coherent configuration, strongly regular graph decomposition.}

\maketitle

\begin{abstract}
 Inspired by some intriguing examples, we study
 uniform association schemes and uniform coherent configurations, including cometric Q-antipodal association
 schemes. After a review of imprimitivity, we show that an imprimitive
 association scheme is uniform if
 and only if it is dismantlable, and we cast these schemes in
 the broader context of certain --- uniform --- coherent
 configurations. We also give a third characterization of uniform
 schemes in terms of the Krein parameters, and derive
 information on the primitive idempotents of such a scheme.

 In the second half of the paper, we apply these results to
 cometric association schemes.
 We show that each such scheme is uniform if and only if it
 is Q-antipodal, and derive results on the parameters of the
 subschemes and dismantled schemes of cometric Q-antipodal schemes. We revisit
 the correspondence between uniform indecomposable three-class
 schemes and linked systems of symmetric designs, and show that
 these are cometric Q-antipodal. We obtain a characterization of
 cometric Q-antipodal four-class schemes in terms of only a few
 parameters, and show that any strongly
 regular graph with a (``non-exceptional") strongly regular decomposition gives rise
 to such a scheme. Hemisystems in generalized quadrangles provide
 interesting examples of such decompositions. We finish with a
 short discussion of five-class schemes as well as a
 list of all feasible parameter sets for cometric Q-antipodal
 four-class schemes with at most six fibres and fibre size at
 most 2000, and describe the known examples.
 Most of these examples are related to groups, codes, and
 geometries.
 \end{abstract}

\section{Introduction}

Motivated by the search for cometric (Q-polynomial) association
schemes, we study uniform association schemes. Cometric
association schemes are the ``dual version'' of
distance-regular graphs (metric schemes), and the latter are
well-studied objects, cf. \cite{bcn, dkt12}. Classical metric schemes
such as Hamming schemes and Johnson schemes are in fact also
cometric. Bannai and Ito \cite[p.\ 312]{banito} conjectured that
for large enough $d$, a primitive $d$-class scheme is metric if
and only if it is cometric. Partly because of this conjecture,
the topic of cometric association schemes was studied mainly in
connection to distance-regular graphs, at least until the end
of last century. An exception to this is the work of Delsarte
\cite{del} (and others building on this) who showed the
importance of cometric schemes in design theory.

This slowly changed when De Caen and Godsil raised the challenging problem of
constructing cometric schemes that are not metric or duals of metric schemes
(cf. \cite[p.\ 234]{godsil}, \cite[Acknowledgments]{mmw}). Around the same time,
Suzuki derived fundamental results on imprimitive cometric schemes
\cite{suzimprim} and on cometric schemes with multiple Q-polynomial orderings
\cite{suztwoq}, but examples of the above type were still missing. In the last
few years, however, there has been considerable activity in the area, with the
first new constructions of cometric (but not metric) schemes given by Martin,
Muzychuk, and Williford \cite{mmw}. For a recent overview of results on
cometric schemes we refer to the survey on association schemes by Martin and
Tanaka \cite{mtanaka}. Very recent is the work of Kurihara and Nozaki
\cite{Kur2011T, KN2012JCTA}, Penttila and Williford \cite{penwil}, and
Suda \cite{suda3, suda1, suda2, Suda2012JCTA, Suda2012pre}.

Meanwhile, in \cite{HigmanCA}--\cite{Huninform}, Higman obtained numerous
results on imprimitive association schemes and coherent configurations.
In his paper on four-class schemes and
triality \cite{Htriality} and also in an unpublished manuscript
\cite{Huninform}, he introduced the concept of uniformity of an
imprimitive scheme, and he mentioned several examples of such
uniform schemes. It turns out that many of these examples are
cometric Q-antipodal. Inspired by this, we work out the concept
of uniformity, and apply it to cometric Q-antipodal schemes.

This paper is organized as follows. We finish this introduction with an intriguing introductory example: the linked
system of partial $\lambda$-geometries that is related to the Hoffman-Singleton graph. This example gives rise to a
cometric Q-antipodal association scheme, and illustrates many of the interesting features we will consider in the
paper. In Section \ref{Sec:background}, we remind the reader of basic background material on association schemes,
focusing in particular on the natural subschemes and quotient schemes of an imprimitive association scheme. The main
results for the first half of the paper are to be found in Sections \ref{sec:uniform} and \ref{Sec:coco}. We first show
in Section \ref{Subsec:dismunif} that the dismantlability property introduced in \cite{mmw} is implied by Higman's
uniformity property \cite{Huninform}. In order to establish the reverse implication, we need to consider a fission of
our uniform association scheme whose adjacency algebra is necessarily non-commutative. So we introduce coherent
configurations at this point to draw out the deeper structure that occurs here. Only at the level of this more detailed
structure do we see the full equivalence of the dismantlable and uniform properties in Theorem \ref{dismunif}. We
finish the first half of the paper with another characterization of the same phenomenon in Section
\ref{Subsec:QHigman}, this time cast in terms of Krein parameters only. We introduce Q-Higman schemes and show that
these, too, are equivalent to uniform schemes. To place the main concepts discussed here in perspective, we summarize
them in the Venn diagram of Figure \ref{venn}.

The second half of the paper returns to the cometric case and explores the
implications of the results discussed above for cometric Q-antipodal schemes.
In Section \ref{sec:cometricQ}, as in Sections \ref{Sec:background} and
\ref{sec:uniform}, we strive to make the paper fairly self-contained; we
include all definitions that are not available in the standard literature. We
show that each cometric scheme is uniform if and only if it is Q-antipodal, and
derive results on the parameters of the subschemes and dismantled schemes of
cometric Q-antipodal schemes. This general discussion of cometric Q-antipodal
schemes is followed by three more detailed sections focusing on such
association schemes with a small number of classes. In Section
\ref{sec:threeclass}, we show that uniform indecomposable three-class schemes
are always cometric Q-antipodal, and that these correspond naturally to linked
systems of symmetric designs. In Section \ref{sec:four}, we study the more
complicated case of four-class schemes. We obtain a characterization of
cometric Q-antipodal four-class schemes in terms of just a few of their
parameters, and show that any strongly regular graph with a
(``non-exceptional") strongly regular decomposition gives rise to such a
scheme. An exciting special case of recent interest is that of hemisystems in
generalized quadrangles. To facilitate future work on such problems, we
generate a list of all feasible parameter sets for cometric Q-antipodal
four-class schemes with at most six fibres and fibre size at most 2000, and
describe the known examples from this table. In the short Section
\ref{sec:five}, we mention some examples of five-class schemes that are
cometric Q-antipodal. The final section, Section \ref{Sec:misc}, collects some
miscellaneous remarks.

As background we refer to Cameron \cite{cameronbook} and Higman
\cite{HigmanCA} for coherent configurations, and to Bannai and
Ito \cite{banito}, Brouwer, Cohen, and Neumaier \cite{bcn},
Godsil \cite{godsil}, and Martin and Tanaka \cite{mtanaka} for
association schemes.

\begin{figure}[h!]
\begin{center}
\includegraphics[width=120mm]{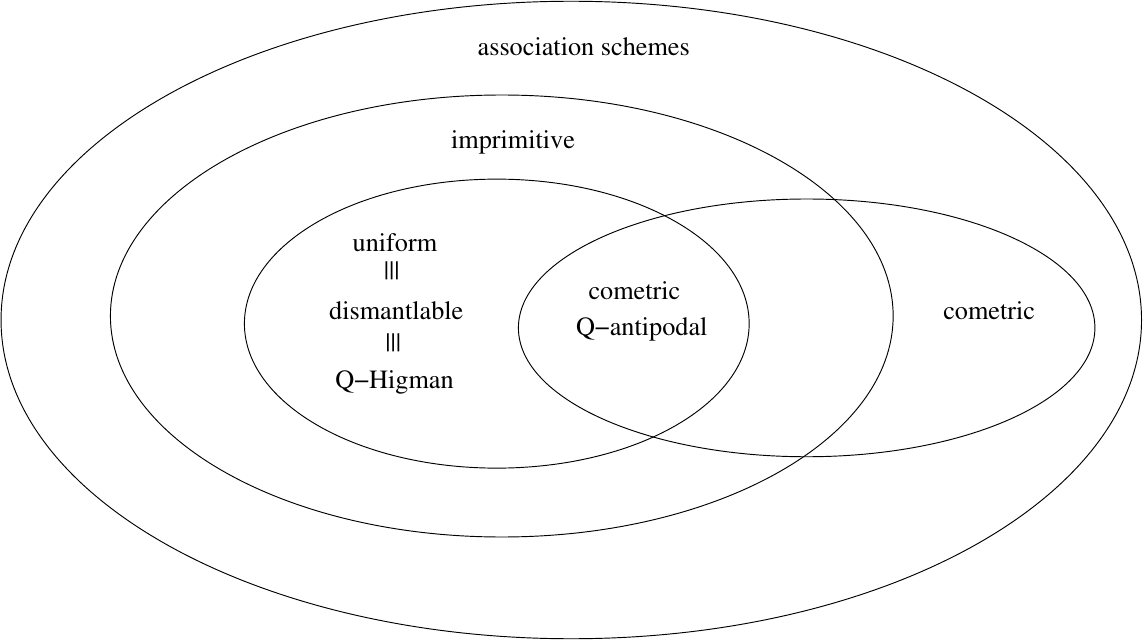}
\caption{Venn diagram of relevant types of association schemes} \label{venn}
\end{center}
\end{figure}

\subsection{A linked system of partial
$\lambda$-geometries related to the Hoffman-Singleton graph: an instructive
example}\label{HOSI}

The maximum size of a coclique in the Hoffman-Singleton graph
is 15. There are 100 cocliques of this size, and it is known
that one can define a bipartite cometric distance-regular graph
$\Gamma$ with diameter four and valency 15 on these 100
cocliques by calling two cocliques adjacent whenever they
intersect in eight vertices, cf. \cite[p.\ 393]{bcn}.
Miraculously, the distance-four graph $\Gamma_4$ of this graph
forms a Hoffman-Singleton graph on each part of the
bipartition. Moreover, the union of $\Gamma$ and $\Gamma_4$ is
the so-called Higman-Sims graph. In fact, in this way it is
clear that the Higman-Sims graph can be decomposed into two
Hoffman-Singleton graphs; here we have a strongly regular
decomposition of a strongly regular graph, in the sense of
Haemers and Higman \cite{HH}. The incidence structure that
$\Gamma$ induces between the two parts of the bipartition is a
so-called strongly regular design as defined by Higman
\cite{Hsrd}, and more specifically a partial $\lambda$-geometry
as defined by Cameron and Drake \cite{camdrake}. Building on a
description of the Hoffman-Singleton graph by Haemers
\cite{Hthesis}, Neumaier \cite{neumaier} describes this partial
$\lambda$-geometry --- and hence the graph $\Gamma$ ---
using the points, lines, and planes of $PG(3,2)$. So far, so good.

Neumaier goes on to describe how $\Gamma$ can be constructed
in the Leech lattice. Using the group $2\cdot U_3(5) \cdot
S_3$, he finds three types of 50 vectors each, and between each
two types of 50 vectors the above partial $\lambda$-geometry.
Moreover, these geometries are linked: we have a linked system
of partial $\lambda$-geometries.

What is going on combinatorially is that one can extend the
distance-regular graph $\Gamma$ by the 50 vertices of the
Hoffman-Singleton graph, by calling a coclique adjacent to a
vertex whenever the coclique contains the vertex. This gives a
30-regular graph on 150 vertices, and it generates a uniform
imprimitive four-class association scheme.  This association
scheme turns out to be cometric too (but it is not metric); in
fact it is Q-antipodal with three fibres of size 50. Here
(again) one of the relations forms a Hoffman-Singleton graph on
each fibre, and between each pair of fibres is the incidence
structure of a partial $\lambda$-geometry (strongly regular
design).

One natural question is whether you can throw in another 50
vertices, and get yet another cometric association scheme.
We address this specific case in Section \ref{smallcases},
and give a general bound on the number of fibres in
Section \ref{Subsec:absbound}.

Higman also gives the above example in his paper on four-class
imprimitive schemes \cite{Htriality}, and in his unpublished
manuscript on uniform schemes \cite{Huninform}. This fairly small
example illustrates most of the central features considered in this
paper and, in our view, the attractive interplay of
combinatorial subjects that one sees in the study of
cometric Q-antipodal association schemes.

\section{Association schemes}
\label{Sec:background}

Our goal in this section is to review briefly the basic definitions
from the theory of association schemes that we will need and to
summarize some necessary material from the theory of imprimitive schemes.
We defer our review of coherent configurations to Section \ref{Sec:coco}
since their role will become clear at that point in the narrative.

\subsection{Definitions}

A (symmetric) $d$-class association scheme $(X,\cR)$
consists of a finite set $X$ of size $v$ and a set $\cR$ of
relations on $X$ satisfying
\begin{itemize}
\item $\cR=\{R_0,\ldots, R_d\}$ is a partition of $X \times X$;
\item $R_0=\Delta_X:=\{(x,x)| x \in X\}$ is the identity relation;
\item $R_i^\top = R_i$ for each $i$,
where $R_i^\top:=\{(x,y) | (y,x) \in R_i\}$;
\item there exist integers $p_{ij}^h$ such that
$$ \left| \left\{ z \in X \, | \,  (x,z) \in R_i \ {\mbox {\rm and}} \ (z,y) \in R_j\right\}
\right| = p_{ij}^h$$ whenever $(x,y) \in R_h$, for each $i,j,h \in
\{0,\ldots , d\}$.
\end{itemize}
The integers $p_{ij}^h$ are called the intersection numbers of the scheme.

The adjacency matrix $A_R$ of a relation $R$ on $X$ is a $v\times v$
$(0,1)$-matrix defined by $(A_R)_{xy}=1$ if $(x,y)\in R$, and
zero otherwise. In this case, we abbreviate by $A_i:=A_{R_i}$
the adjacency matrix of relation $R_i$
and consider $\cA:=\langle A_i|i=0,\dots,d \rangle$. Then this
vector space is a $(d+1)$-dimensional commutative algebra of
symmetric matrices; this is called the Bose-Mesner algebra of
the association scheme.  Such an algebra admits a basis of
pairwise orthogonal primitive idempotents (a nonzero idempotent $E$ of
$\cA$ is called primitive if $AE$ is proportional to $E$ for
each $A\in\cA$). We denote these by $E_0, E_1, \ldots, E_d$
with the convention that $E_0 = \frac{1}{v}J$ where $J= \sum_i
A_i$ is the all-ones matrix. The first and second eigenmatrices
of the scheme are denoted by $P$ and $Q$, respectively, and are
defined by the change-of-basis equations
$$ A_i = \sum_j P_{ji} E_j \qquad \text{and} \qquad E_j = \frac{1}{v} \sum_i Q_{ij} A_i.$$

The algebra $\cA$ is also closed under entrywise (Schur-Hadamard)
multiplication $\circ$ of matrices because $A_i \circ A_j  =
\delta_{ij} A_i$. (We call the $(0,1)$-matrices $A_i$ the primitive Schur
idempotents of $\cA$.) The (nonnegative) Krein parameters (or dual
intersection numbers) $q_{ij}^h$ are the structure constants for
this multiplication with respect to the basis of primitive
idempotents:
$$ E_i \circ E_j = \frac{1}{v} \sum_h q_{ij}^h E_h . $$
We abbreviate $v_i := P_{0i}=p^0_{ii}$ and call this the  $i^{\rm th}$ valency; likewise, $m_j := Q_{0j}=q^0_{jj}$ is
called the $j^{\rm th}$ multiplicity of the scheme.

\subsection{Metric schemes and cometric schemes}

The association scheme $(X,\cR)$ is called metric (or
``P-polynomial'') if there exists an ordering
$R_0,R_1,\ldots,R_d$ of the relations for which
\begin{itemize}
\item $p_{ij}^h = 0$ whenever $0 \le h < |i-j|$ or $i + j < h $, and
\item $p_{ij}^{i+j} > 0$ whenever $p_{ij}^{i+j}$ is defined.
\end{itemize}
An ordering with respect to which these properties hold is
called a P-polynomial ordering. In this case, $R_i$ can be
interpreted as the distance-$i$ relation in the simple graph
$(X,R_1)$ which is necessarily distance-regular. Metric schemes
with given P-polynomial orderings are in one-to-one
correspondence with distance-regular graphs.

The association scheme $(X,\cR)$ is called cometric (or
``Q-polynomial'') if there exists an ordering of the primitive
idempotents $E_0,E_1,\ldots,E_d$ for which
\begin{itemize}
\item $q_{ij}^h = 0$ whenever $0 \le h < |i-j|$ or $i + j < h$, and
\item $q_{ij}^{i+j} > 0$ whenever $q_{ij}^{i+j}$ is defined.
\end{itemize}
It is well known (cf. \cite[Prop. 2.7.1]{bcn}) that to check
that a scheme is cometric it suffices to check these properties
for $i=1$. An ordering with respect to which these hold is
called a Q-polynomial ordering, and $E_1$ is called a
Q-polynomial generator. There is no known simple combinatorial
or geometric interpretation of the cometric property. Suzuki
\cite{suztwoq} showed that, while it is possible to have two
distinct Q-polynomial orderings, there can be no more than two
such orderings for a given association scheme, with the
exception of the cycles.  Several important families of
association schemes, such as the Hamming schemes and Johnson
schemes, are both cometric and metric. But our study here does
not assume the metric property at all.

Let $c_i^*:=q^i_{1,i-1}, a_i^*:=q^i_{1i}$, and
$b_i^*:=q^i_{1,i+1}$. Then  $c_i^* +  a_i^* + b_i^* = q^0_{11}$ and
the Krein array of the cometric
association scheme is defined as
$$\{b_0^*,b_1^*,\dots,b_{d-1}^*;c_1^*,c_2^*,\dots,c_d^*\}.$$
Using the Krein array, we define a sequence of orthogonal polynomials $q_j$,
$j=0,1,\dots,d+1$ by $q_0(x)=1$, $q_1(x)=x$, and the three-term recurrence
$xq_j(x)=c_{j+1}^*q_{j+1}(x)+a_j^*q_j(x)+b_{j-1}^*q_{j-1}(x)$, where we let
$c_{d+1}^*:=1$. It follows that $vE_j=q_j(vE_1)$, $j=0,1,\dots,d$ where matrix
multiplication is entrywise (and hence the empty product is $J$). Moreover,
because $v E_1 \circ E_d = b_{d-1}^* E_{d-1} + a_d^* E_d$, we have that the
roots of $q_{d+1}(x)$ are precisely $Q_{i1}$ for $i=0,\ldots,d$. It is now easy
to see that, whenever $E_1$ is a Q-polynomial generator for the Bose-Mesner
algebra, column one of the matrix $Q$ has $d+1$ distinct entries.

\subsection{Imprimitive schemes}
\label{Subsec:imprim}

The association scheme $(X,\cR)$ with Bose-Mesner algebra $\cA$, adjacency matrices $A_i, i=0,1,\dots,d$, and primitive
idempotents $E_j, j=0,1,\dots,d$ is called imprimitive if at least one of its nontrivial relations is disconnected (as
a graph). It was first shown by Cameron, Goethals, and Seidel \cite{cgs} (and not hard to verify, cf. \cite[Thm.~9.3,
Thm.~4.6]{banito}) that imprimitivity is equivalent to each of the following properties:
\begin{itemize}
\item there is a set $\cI$ with $\{0\} \subsetneq \cI
    \subsetneq \{0,1,\dots,d\}$ such that $\langle A_i | i
    \in \cI \rangle$ is a matrix subalgebra of $\cA$;
\item there is a set $\cJ$ with $\{0\} \subsetneq \cJ
    \subsetneq \{0,1,\dots,d\}$ such that $\langle E_j | j\in
    \cJ \rangle$ is a $\circ$-subalgebra of $\cA$;
\item there is a matrix\footnote{This matrix is given
by Equation \eqref{Eimprim}.} $E \in \cA$, not $0$, $I$, or $J$,
    such that $E^2=n E$ and $E \circ E = E$ for some $n$;
\item the matrix $E_j$ has repeated columns for some $j>0$.
\end{itemize}

\noindent For an imprimitive scheme, the sets $\cI$ and $\cJ$
may not be unique, however the various index sets $\cI$ and $\cJ$ are
paired by the following equation:
\begin{equation}
\label{Eimprim}
\sum_{i \in \cI} A_i=n\sum_{j \in \cJ} E_j=I_{w}\otimes J_n
\end{equation}
for some choice of ordering of the vertices.
Thus the $v$ vertices are partitioned into $w$
fibres of size $n$. Like $\cI$ and $\cJ$, this partitioning $\cF$
into fibres --- the so-called imprimitivity system ---  may not
be unique, but each of $\cI$, $\cJ$, $\cF$ is well-defined given
any other one of the three. In the remainder of the paper we will
always assume however that $\cI$, $\cJ$, and the imprimitivity system
are fixed and given, unless mentioned otherwise. Of the equivalent
statements of imprimitivity, the last one could be explained as
``dual imprimitivity". In fact, in this case each of the
matrices $E_j$, $j \in \cJ$ is constant on each fibre $U \in \cF$
(i.e., columns $x$ and $y$ of $E_j$ are identical when $x,y\in U$).
This is analogous to the fact that
each relation $R_i, i \in \cI$ is disconnected.

It easily follows that on each fibre $U\in \cF$, there is an association
scheme --- a so-called subscheme --- induced by the relations
indexed by $\cI$. In fact, the intersection numbers
$\tilde{p}^h_{ij}$ of the subscheme are the same as the
corresponding ones in the original scheme, i.e.,
$$\tilde{p}^h_{ij}=p^h_{ij}, i,j,h \in \cI.$$
To put things differently,
$\cB:=\langle A_i | i \in \cI \rangle$
is a Bose-Mesner subalgebra of the
Bose-Mesner algebra $\cA$ (i.e., $\cB$ is a subalgebra under both
ordinary and entrywise multiplication). For later purpose, we
define a linear (projection) operator $\pi:\cA \rightarrow \cA$
by
\begin{equation}
\label{Epi}
\pi(A)=A \circ (I_w \otimes J_n)
\end{equation}
for $A \in \cA$.
It is clear that $\pi(A \circ A')=\pi(A) \circ \pi(A')$ for all
$A,A' \in \cA$. Because the map $\pi$ sends $A = \sum_{i=0}^d
c_i A_i$ to $\sum_{i\in \cI} c_i A_i$, it is also clear that
$\pi(\cA)=\cB$.
Note also that $\cB$ is
a $\circ$-ideal in $\cA$, because if $A \in \cA$ and $B\in
\cB$, then $A \circ B =A \circ \pi(B)=\pi(A) \circ B \in \cB$.

Each imprimitivity system also gives us a quotient association scheme.
Dual to $\cB$, consider
$$\cC:=\langle E_j | j \in \cJ \rangle =\{A(I_w \otimes J_n) | A \in \cA\};$$
this is also a Bose-Mesner subalgebra of $\cA$.
It is the image of $\cA$ under the projection $\pi^\ast$
which sends $A = \sum_{j=0}^d c_j E_j$ to
$\pi^\ast(A):=\frac{1}{n}A(I_w \otimes J_n) = \sum_{j\in
\cJ} c_j E_j$.  Each Schur idempotent of $\cC$ must be a sum of
certain $A_i$ and if $A=\sum_{i\in \cH} A_i$ satisfies $A =
\frac{1}{n}A(I_w \otimes J_n)$, then $A_{xy} = A_{x'y'}$
whenever $x$ is in the same fibre as $x'$, and $y$ is in the
same fibre as $y'$.
So, for each $C \in \cC$,
there exists a well-defined $w\times w$ matrix $\iota(C)$ satisfying
$$C = \iota(C) \otimes J_n . $$
It is not hard to verify that the set $\left\{ \iota(C)|C \in \cC
\right\}$ is a Bose-Mesner algebra also; this gives an association
scheme --- the so-called quotient scheme --- on the set of fibres.
In this case, the Krein parameters of this quotient scheme are the
same as the corresponding ones in the original scheme (cf.
\cite[Sec.~2.4]{bcn}). For completeness we mention that Rao,
Ray-Chaudhuri, and Singhi \cite{rao} obtained results on the
composition factors of imprimitive schemes.

For the topic of this paper --- uniform schemes and, later, cometric Q-antipodal schemes --- our main interest is in
the relation between the scheme and its subschemes. The corresponding quotient scheme is in this case trivial, that is,
a one-class scheme corresponding to a complete graph. The relationship between the scheme and its subschemes and
quotient schemes is essentially worked out by Bannai and Ito \cite[Thm.~II.9.9]{banito} (see also \cite[Sec.~2.4]{bcn}
for some information on the relation between the parameters). However, to get a better understanding of what is going
on, we include some of their arguments and results (and those of others) applied to subschemes here. (Moreover, Bannai
and Ito treated the dual case, which, even though it is analogous, may sometimes be confusing.) By doing this, we
derive in Lemma~\ref{subkrein} another (and new, as far as we know) relation between the parameters.

Following Bannai and Ito, we define the relation $\sim^*$ on the
index set $\{0,1,\dots,d\}$ (indexing the primitive idempotents) by
$$i \sim^* j :\Leftrightarrow q^h_{ij} \neq 0 \text{~for some~}h \in
\cJ.$$

\begin{lemma} The relation $\sim^*$ is an equivalence relation.
\end{lemma}

\begin{proof}
If $i \sim^* j \sim^* l$, say $q^h_{ij} \neq 0$ and $q^{h'}_{jl} \neq 0$ with $h,h' \in
\cJ$, then by using a standard identity
(cf. \cite[Prop. II.3.7(vii)]{banito}, \cite[Lem.~2.3.1(vi)]{bcn}) and the fact that
$q^{h''}_{hh'}=0$ if $h'' \notin \cJ$, we obtain that
$$\sum_{h'' \in \cJ}q^l_{ih''}q^{h''}_{hh'}=\sum_{h''=0}^d q^l_{ih''}q^{h''}_{hh'}=\sum_{j'=0}^d
q^{j'}_{ih}q^{l}_{j'h'}\geq q^{j}_{ih}q^{l}_{jh'}>0,$$ and it follows that for some $h''
\in \cJ$ we have $q^{h''}_{il} \neq 0$, i.e., $i \sim^* l$.
\end{proof}

\noindent One of the equivalence classes of this relation must
be $\cJ=:\cJ_0$, and we label the others by $\cJ_1,\dots,\cJ_e$.

\begin{example}
In the linked system of geometries described in the introduction, we
obtained a four-class imprimitive association scheme on $150$
vertices. In that example, if we use the $Q$-polynomial ordering of
the eigenspaces, the relation $\sim^*$ has equivalence classes
$\cJ_0 = \{0,4\}$, $\cJ_1 = \{1,3\}$ and $\cJ_2 = \{2\}$, which, as
we shall later see, is indicative of the Q-antipodal case.
\end{example}

\noindent Now we claim that the idempotents $$F_j:=\sum_{j' \in
\cJ_j}E_{j'},\qquad j=0,1,\dots,e$$ are the primitive idempotents of $\cB$,
and hence by restricting these to a fibre we obtain the primitive
idempotents of the subscheme on that fibre. To prove this claim, and
to obtain a useful relation between the Krein parameters of $\cA$
and $\cB$, we define the nonnegative parameter
\begin{equation}
\label{Erhodef}
\rho^i_j:=\sum_{h\in \cJ}q^i_{jh},
\end{equation}
and note that $\rho^i_j
\neq 0$ if and only if $i \sim^* j$. We abbreviate
$\rho^j_j=:\rho_j$.

\begin{lemma}
\label{idempotents}
The primitive idempotents of $\cB$ are $F_j$,
$j=0,1,\dots,e$, so $\cB$ has dimension $|\cI|=e+1$.
 Moreover, if $j' \in \cJ_{j}$, then
$\pi(E_{j'})=\frac{\rho_{j'}}{w}F_{j}$.
\end{lemma}

\begin{proof} We first note that each primitive idempotent of $\cB$ is a sum of
primitive idempotents of $\cA$, and because $\sum_{j=0}^{d}E_j=I \in
\cB$, each $E_j$ appears in exactly one such sum.  Then for each
$j=0,1,\dots,d$, we use \eqref{Epi}, \eqref{Eimprim}, and
\eqref{Erhodef} to find
$$\pi(E_j)=n \sum_{h \in \cJ} E_j \circ E_h =\frac{1}{w}\sum_{i=0}^d \rho^i_j E_i.$$
Thus, if $\cH \subseteq \{0,\ldots,d\}$ and
$F:=\sum_{j \in \cH}E_j$ is any idempotent of $\cB$,
then
$$F=\pi(F)= \sum_{j \in \cH} \pi(E_j)=\frac{1}{w} \sum_{i=0}^d \sum_{j \in \cH} \rho^i_j E_i.$$
This implies that if $i \notin \cH$, then $\sum_{j \in \cH}
\rho^i_j =0$, i.e., if $i \notin \cH$ and $j \in \cH$, then $i
\nsim^* j$, which proves that $\cH$ is a union of equivalence
classes of $\sim^*$.

On the other hand, take any $0\le j\le d$ and consider the primitive
idempotent $F:=\sum_{h \in \cH}E_h$ for which $j \in \cH$. Because
$\frac{1}{w}\sum_{i=0}^d \rho^i_j E_i =
\pi(E_j)\in \cB$, it is a linear combination of
primitive idempotents of $\cB$ with a nonzero coefficient for $F$
because $\rho^j_j>0$. So, if
$h \in \cH$, then $\rho^h_j>0$, which shows
that $h$ and $j$ are in the same equivalence class. We may
therefore conclude that $\cH$ is an equivalence class of
$\sim^*$.

Thus, the primitive idempotents of $\cB$ are $F_j$,
$j=0,1,\dots,e$. For $j' \in \cJ_{j}$, it then also follows that
$\pi(E_{j'})=\frac{1}{w}\sum_{i \sim^*j'} \rho^i_{j'} E_i$ is a
multiple of one of these idempotents. So
$\rho^i_{j'}=\rho_{j'}$ for all $i \sim^*j'$, and
$\pi(E_{j'})=\frac{\rho_{j'}}{w}F_{j}$.
\end{proof}

\noindent
By working out the products $F_i \circ F_j$, the Krein parameters
$\tilde{q}^h_{ij}$ of the subscheme can now be easily expressed
in terms of those of the original scheme as
$$\tilde{q}^h_{ij}=\frac{1}{w} \sum_{i' \in \cJ_i, j' \in
\cJ_j} q^{h'}_{i'j'},$$ for each $h' \in \cJ_h$. Moreover, it
follows that the eigenmatrices $\tilde{P}$ and $\tilde{Q}$ of
the subscheme are given by

\begin{gather}\label{EPQtilde}
\begin{aligned}
&\tilde{P}_{ji} =  P_{j'i}, \qquad  i \in \cI, j' \in \cJ_j,   j=0,1,\dots,e; \\
&\tilde{Q}_{ij} = \frac{1}{w}\sum_{j' \in \cJ_j}Q_{ij'}, \qquad i \in \cI,  j=0,1,\dots,e.
\end{aligned}
\end{gather}
However, the second part of Lemma \ref{idempotents} can be used
to get another useful expression of the Krein parameters of the
subschemes.

\begin{lemma}\label{subkrein} If $i' \in \cJ_i$, $j' \in \cJ_j$, then
$$\tilde{q}^h_{ij}=\frac{1}{\rho_{i'} \rho_{j'}} \sum_{h' \in \cJ_h} \rho_{h'} q^{h'}_{i'j'} .$$
\end{lemma}

\begin{proof} Let $i' \in \cJ_i$, $j' \in \cJ_j$, then
\begin{equation*}
\rho_{i'}
\rho_{j'}F_i \circ F_j = w^2 \pi(E_{i'} \circ E_{j'}) =
\frac{w^2}{v} \sum_{h'=0}^d
q^{h'}_{i'j'}\pi(E_{h'})=\frac{1}{n} \sum_{h=0}^e \sum_{h' \in
\cJ_h} \rho_{h'} q^{h'}_{i'j'} F_h,
\end{equation*}
which was to be proven.
\end{proof}

\section{Uniform imprimitive schemes}
\label{sec:uniform}

So far, we have given a selective review of imprimitive association
schemes, focusing on the eigenspaces and the Krein parameters of subschemes.
Exploring imprimitivity further, the main goal of this section
is to reconcile the concept of dismantlable
association scheme introduced in \cite{mmw} with the concept of uniform
association scheme introduced earlier in \cite{Huninform}.

\subsection{Dismantlability and uniformity}
\label{Subsec:dismunif}

Besides the usual subschemes on each fibre, it was proven in
\cite[Thm. 4.7]{mmw} that a cometric Q-antipodal scheme has
so-called dismantled schemes on each union of fibres. To
generalize this result, and to obtain more information on these
dismantled schemes in the subsequent sections, we first define
the following.

For a subset $Y$ of the vertices, let $I^Y$ be the $v \times v$
diagonal $(0,1)$-matrix with $(I^Y)_{xx}=1$ if and only if $x
\in Y$. For a matrix $M$, we let $$M^{YZ}:=I^Y M I^Z$$ for
subsets $Y$ and $Z$. Put differently, $M^{YZ}$ is the $v \times
v$ matrix containing the submatrix $M_{YZ}$, and that is zero
everywhere else. Algebraically, in most of the following it
turns out to be more convenient to work with the matrices
$M^{YZ}$ than with the usual submatrices $M_{YZ}$, although
essentially they are the same. For a relation $R$ we define
related notation
$$R^{YZ}:=R \cap (Y \times Z).$$
In case $Y=Z$, we often use shorthand
notation $M^{Y}:=M^{YY}$ and $R^{Y}:=R^{YY}$. For a set
$\cM$ of matrices, we let $\cM^Y:=\{M^Y | M
\in \cM\}$ and for a set
$\cR$ of relations, we let $\cR^Y:=\{R^Y \neq \emptyset | R
\in \cR\}$.

\begin{definition}An imprimitive association scheme $(X,\cR)$
is called
dismantlable if $(Y, \cR^Y)$ is an association scheme for each
union $Y$ of fibres. In this case, the association scheme $(Y,
\cR^Y)$ is called a dismantled scheme on $Y$, if $Y$ is the
union of at least two fibres.
\end{definition}

\noindent This definition first appears in \cite{mmw} where
the structure of cometric Q-antipodal association schemes
is considered.  We shall see in Corollary \ref{Cdismantledallsame}
that two dismantled schemes $(Y, \cR^Y)$ and $(Y', \cR^{Y'})$
of $(X,\cR)$ with $|Y|=|Y'|$ always have the same parameters.

Bipartite schemes, i.e., imprimitive schemes with
two fibres, are trivially dismantlable. Other examples of
dismantlable schemes are the so-called uniform association
schemes, as defined by Higman in his paper on four-class
imprimitive schemes \cite{Htriality} and more generally in an
unpublished manuscript \cite{Huninform}. Informally speaking,
an imprimitive scheme is uniform if the intersection numbers
are divided uniformly over the fibres whereas, in the general
case, only the valencies enjoy this property.

To define uniform schemes precisely, we first introduce a bit of
notation. Consider an imprimitive scheme with a trivial quotient
scheme, i.e., where the quotient is a complete graph. As in Equation
\eqref{Eimprim}, let $\cI$ denote the indices of relations that
occur in the subschemes. For fibres $U$ and $V$, we denote by
$\cI(U,V)$ the index set of relations that occur between $U$ and
$V$; so $A_i^{UV}$ is nonzero precisely if $i \in \cI(U,V)$. Because
we are assuming that the quotient is a complete graph, $\cI(U,V)$
equals $\cI$ if $U=V$, and $\cI(U,V)=\overline{\cI}$ (the complement
of $\cI$) if $U \neq V$.

\begin{definition}
\label{Dunifscheme}
An imprimitive association scheme is called uniform if its
quotient scheme is trivial, and if there are integers
$a_{ij}^h$ such that for all fibres $U,V,$ and $W,$ and $i \in
\cI(U,V)$, $j \in \cI(V,W)$, we have
\begin{equation}\label{coh_conf_prod}
A_i^{UV} A_j^{VW} = \sum_h a_{ij}^h A_h^{UW}.
\end{equation}
\end{definition}

\noindent It is easily seen that in this case $p_{ij}^h = a_{ij}^h$ if $i \in
\cI$ or $j \in \cI$, $p_{ij}^h = (w-1) a_{ij}^h$ if $i,j \notin \cI$ and $h \in
\cI$, and $p_{ij}^h= (w-2) a_{ij}^h$ if $i,j,h \notin \cI$, i.e., the
intersection numbers are divided uniformly over the relevant fibres. Note that
bipartite schemes are trivially uniform. But in general, an
 antipodal distance-regular cover of a complete graph, while
 having a complete quotient, is not uniform; for example, two
 adjacent vertices in the icosahedron have only two common
 neighbors so $p_{11}^1$ is not divisible by $w-2=4$. However, any imprimitive $d$-class
association scheme with only one relation across fibres (a complete
multipartite graph) is uniform. Such a scheme can easily be constructed as a
wreath product scheme \cite[p.\ 44]{weis}, \cite[p.\ 69]{bailey} of a trivial scheme
and an arbitrary scheme. Also the tensor product \cite[p.\ 44]{weis} of a
one-class scheme and an arbitrary scheme is uniform. (This is also called the
``direct product'' \cite[p.\ 62]{bailey}.) In this paper, we call a scheme
decomposable if it is has the same parameters as a wreath product or tensor
product scheme.

\begin{theorem}
\label{uniformdismantle}
A uniform scheme is dismantlable.
Any dismantled scheme of a uniform scheme is also uniform.
\end{theorem}

\begin{proof} These  claims follow in a straightforward way from
the definition of a uniform scheme.
\end{proof}

\noindent In Section \ref{sec:uniformcoco} we will show the
converse of this proposition, namely that every dismantlable
scheme is uniform.

\subsection{Linked systems and triality}

In Section \ref{HOSI} we described what we (and Neumaier
\cite{neumaier}) called a linked system of partial
$\lambda$-geometries. This linked system is in fact a uniform
association scheme with three fibres of size 50. The term
linked system was coined by Cameron \cite{cameronlinked} for
linked systems of symmetric designs (see also Section
\ref{sec:threeclass}).

\begin{example}
There are three non-isomorphic $(16,6,2)$ symmetric block designs. Each
incidence structure gives us a three-class bipartite association scheme
with two fibres of size sixteen. But only one of these can be extended to a
linked system of symmetric designs with eight fibres of size sixteen. This
is a uniform cometric scheme on $128$ vertices and is the first example in
an infinite family which arises from the Kerdock codes \cite{cs} (see
also \cite{noda,mmw}).
\end{example}

\noindent Following Neumaier, and also Cameron and Van
Lint \cite{cvl} (see Section \ref{sec:vls}), we will use the
term linked system informally for the combinatorial structure
underlying a uniform association scheme. Note also that Higman
\cite{Htriality} mentions the term ``system of uniformly linked
strongly regular designs''.  We will now describe an
infinite family of such systems, which we refer to as
Higman's ``triality schemes''.

\begin{example}
\label{Ex-triality} The dual polar graph $D_4(q)$ is a cometric bipartite
distance-regular graph with diameter four defined on the $2(q+1)(q^2+1)(q^3+1)$ maximal isotropic
(four-dimensional) subspaces in $GF(q)^8$ with a quadratic form of Witt index
4. One can extend this graph by a third fibre containing the
$(q+1)(q^2+1)(q^3+1)$ isotropic one-dimensional subspaces, where a four-dimensional subspace
is adjacent to the one-dimensional subspaces that it contains. This extended graph generates
a uniform four-class association scheme that is cometric Q-antipodal. Higman
\cite{Htriality} explains how this scheme is obtained from classical triality
related to the group $O^+_8(q)$, and also how some other sporadic examples,
such as the one in Section \ref{HOSI}, have a triality related to some group.
Higman also mentions that related to these examples are certain coherent
configurations.
\end{example}

\section{Coherent configurations and uniformity}
\label{Sec:coco}

To understand uniformity better, we will need to recall certain combinatorial
structures that are more general than association schemes.
As we will see, a (symmetric) $d$-class association scheme can be viewed as
a homogeneous coherent configuration of rank $d+1$ in which all relations are
symmetric.

\subsection{Definitions and algebraic automorphisms}
\label{Subsec:coco}

A coherent configuration is a pair $(X,\cS)$ consisting of a finite set $X$ of
size $v$ and a set $\cS$ of binary relations on $X$ such that
\begin{itemize}
\item $\cS$ is a partition of $X \times X$;
\item the diagonal relation $\Delta_X$ is the union of some relations in
    $\cS$;
\item for each $R\in\cS$ it holds that $R^\top \in \cS$;
\item there exist integers $p_{ST}^R$ such that
$$ \left| \left\{ z \in X | (x,z) \in S \ {\mbox {\rm and}} \ (z,y) \in T\right\}
\right| = p_{ST}^R$$ whenever $(x,y) \in R$, for each $R,S,T\in\cS$.
\end{itemize}
The relations of $\cS$ are called basic relations of the
configuration. A basic relation $R$ is called a diagonal
relation if $R\subseteq \Delta_X$. Each diagonal relation is of
the form $\Delta_U$ for some $U\subseteq X$. Because the
relations of $\cS$ form a partition of $X \times X$, the
diagonal relations of $\cS$ form a partition of $\Delta_X$.
Thus there exists a uniquely determined partition of $X$ into a
set $\cF_{\cS}$ of $w$ fibres such that $\Delta_{U}
\in \cS$ for each $U \in \cF_{\cS}$. The numbers $v=|X|$ and
$|\cS|$ are called the order and the rank of the configuration,
respectively.

Given $R\in\cS$ and $x\in X$ we define $R(x):=\{y\in
X\,|(x,y)\in R\}$. For any basic relation $R$ we define its
projections onto the first and second coordinates as
$\pr_1(R):=\{x\in X\,|\,R(x)\neq\emptyset\}$ and
$\pr_2(R):=\pr_1(R^\top)$. One can show that these projections
are fibres. So, each basic relation $R$ is contained in
$\pr_1(R) \times \pr_2(R)$. We write $\cS^{UV}$ for the set of
all basic relations $R\in\cS$ with $\pr_1(R)=U, \pr_2(R)=V$,
and $r_{UV}:=|\cS^{UV}|$. Note that $r_{UV}=r_{VU}$ and
$|\cS|=\sum_{U,V}r_{UV}$. The $w\times w$ integer symmetric
matrix $(r_{UV})$ is called the type of the configuration.

The last axiom
of the definition of coherent configuration implies that
$$A_SA_T=\sum_Rp_{ST}^RA_R.$$
It thus follows that the vector subspace of $M_X(\C)$ spanned by
the adjacency matrices $A_R$, $R\in\cS$ is a subalgebra of the
full matrix algebra $M_X(\C)$. It also explains why the
intersection numbers $p_{ST}^R$ are sometimes called structure constants.
The subalgebra is called the adjacency algebra of $\cS$ and
will be denoted by $\C[\cS]$. This algebra has the following
properties:
\begin{itemize}
\item it is closed with respect to (ordinary) matrix multiplication;
\item it is closed with respect to entrywise (Schur-Hadamard)  multiplication $\circ$;
\item it is closed with respect to transposition ${}^\top$;
\item it contains the identity matrix $I$ and the all-ones
    matrix $J$.
\end{itemize}
Any subspace of $M_X(\C)$ which satisfies these conditions is
called a coherent algebra. There is a one-to-one correspondence
between coherent configurations on $X$ and coherent algebras in
$M_X(\C)$, i.e., each coherent algebra is the adjacency algebra
of a uniquely determined coherent configuration.

An algebraic automorphism of $\cS$ is a permutation $\sigma \in \sym(\cS)$
which preserves the structure constants, that is, $p_{ST}^R =
p_{\sigma(S)\sigma(T)}^{\sigma(R)}$ for all $R,S,T\in\cS$ (an algebraic
automorphism of an association scheme is also called a pseudo-automorphism, cf.
\cite{ikuta}). One can extend such a $\sigma$ to a linear map from $\C[\cS]$
into itself by setting $\sigma(\sum_{R\in\cS}\alpha_R
A_R):=\sum_{R\in\cS}\alpha_RA_{\sigma(R)}$. This yields an automorphism of the
adjacency algebra; the linear map defined in this way preserves the ordinary
matrix product, Schur-Hadamard product, and matrix transposition, i.e.,
$\sigma(AB)=\sigma(A)\sigma(B), \sigma(A \circ B)=\sigma(A) \circ \sigma(B)$,
and $\sigma(A^{\top})=\sigma(A)^{\top}$ for all $A,B \in \C[\cS]$. Vice versa,
each permutation $\sigma$ which preserves these three operations is an
algebraic automorphism of $\cS$.

The algebraic automorphisms of $\cS$ form a group (which is a
subgroup of $\sym(\cS)$), which will be denoted by
$\aaut(\cS)$. Any subgroup $G\leq\aaut(\cS)$ gives rise to a
fusion configuration $\fuse{\cS}{G}$ whose basic relations are
$\cup_{R \in O}R$, $O \in \Omega$, where $\Omega$ is the set of
orbits of $\cS$ under the action of $G$.
The adjacency algebra of $\fuse{\cS}{G}$ can be characterized
as the subspace of $\C[\cS]$ consisting of all $G$-invariant
elements of $\C[\cS]$.

The matrices $I^{U}, U \in \cF_{\cS}$ are the only idempotent
matrices of the standard basis $\{ A_R | R \in \cS\}$ of $(X,\cS)$.
Therefore any algebraic automorphism $\sigma$ of $\cS$ permutes
these diagonal matrices, hereby also inducing a permutation $U
\mapsto \sigma(U)$ on the set of fibres. So, instead of
$\sigma(I^{U})$, we could also write $I^{{\sigma(U)}}$.

If $G\leq\aaut(\cS)$ acts transitively on the set of fibres, then
$\fuse{\cS}{G}$ is homogeneous, that is, it is a coherent
configuration with one fibre, or in other words, a
--- possibly nonsymmetric
--- association scheme.

\subsection{Uniformity in coherent configurations}
\label{sec:uniformcoco}

We now make a fundamental observation about uniform
association schemes. Consider such a scheme $(X,\cR)$, with
related (generic) notation as above. It follows immediately
from \eqref{coh_conf_prod} that the set of relations
$\cS:=\{R_i^{UV}| i \in \cI(U,V); \,  U,V \in \cF\}$ forms a
coherent configuration, with the same fibres as those of the
association scheme, i.e., $\cF_{\cS}=\cF$. Moreover, any
$\sigma \in \sym(\cF_{\cS})$ acts as a permutation on $\cS$ by
$\sigma(R^{UV}_i):=R^{\sigma(U)\sigma(V)}_{i}$, for $i\in
\cI(U,V)=\cI(\sigma(U),\sigma(V))$. In this way, $\sigma$ is an
algebraic automorphism of $\cS$, because if $i\in \cI(U,V)$,
$j\in \cI(V,W)$, $h\in \cI(U,W)$, then
$$p_{R^{UV}_i
R^{VW}_j}^{R^{UW}_h} = a^h_{ij}= p_{R^{\sigma(U)\sigma(V)}_i
R^{\sigma(V)\sigma(W)}_j}^{R^{\sigma(U)\sigma(W)}_h}=p_{\sigma(R^{UV}_i)\sigma(
R^{VW}_j)}^{\sigma(R^{UW}_h)}.$$
Moreover, the fusion scheme
$\fuse{\cS}{\sym(\cF_{\cS})}$ is the association scheme that we
started from. These observations are the motivation for the
definition of a uniform coherent configuration. But first we
need a little more terminology. We say that two triples $(U,V,W)$
and $(U',V',W')$ of fibres have the ``same type'' if and only if
there is a permutation $\sigma$ of the fibres such that
$\sigma((U,V,W))=(U',V',W')$.

\begin{definition}
A coherent configuration $(X,\cS)$ with at least two fibres is
called uniform if there are complementary sets of indices
$\cI_{\cS} \ni 0$ and $\overline{\cI_{\cS}}$ of sizes
$e_{\cS}+1$ and $\ell_{\cS}$ (say), respectively, such that the
basic relations $R\in \cS$ can be relabeled as $R=S^{UV}_i$
($U=\pr_1(R)$, $V=\pr_2(R)$, $i\in \cI_\cS \cup
\overline{\cI_{\cS}}$) such that
\begin{itemize}
\item $S^{UU}_0 = \Delta_U$ for each fibre $U$;
\item $\cS^{UU}=\{S^{UU}_i | i\in \cI_{\cS}\}$ for each
    fibre $U$ and $\cS^{UV}=\{S^{UV}_i | i\in
    \overline{\cI_{\cS}}\}$
for all
    fibres $U\neq V$;
\item $(S^{UV}_i)^\top = S^{VU}_i$ for all fibres $U\neq
    V$;
\item for any two triples $(U,V,W)$ and $(U',V',W')$ of the same
    type and any $i\in \cI_{\cS}(U,V)$, $j\in \cI_{\cS}(V,W)$,
    $h\in \cI_{\cS}(U,W)$, it holds that
\begin{equation}\label{coco_uniform}
 p_{S^{UV}_i S^{VW}_j}^{S^{UW}_h} =  p_{S^{U'V'}_i S^{V'W'}_j}^{S^{U'W'}_h}.
\end{equation}
\end{itemize}
\end{definition}

\noindent In this definition $\cI_{\cS}(U,V)$ is defined in the same way as
before: it equals $\cI_{\cS}$ if $U=V$, and $\overline{\cI_{\cS}}$ otherwise.
Without loss of generality we will assume that $\cI_{\cS} \cup
\overline{\cI_{\cS}} = \{0,\dots,e_{\cS}+\ell_{\cS}\}$.

It is clear from the above observations that from a uniform
association scheme one obtains a uniform coherent configuration
with $\cF_{\cS}=\cF$, $\cI_{\cS}=\cI$,
$\overline{\cI_{\cS}}=\overline{\cI}$, $e_{\cS}=e$,
$\ell_{\cS}=d-e$, and $S_i^{UV}=R_i^{UV}$.

Conversely, given a uniform coherent configuration, any
permutation $\sigma$ of the fibres acts --- just as in Section
\ref{Subsec:coco} --- as an algebraic automorphism of $\cS$, by
\eqref{coco_uniform}. Thus, the relations $R_i:=\cup_{U,V}
S_i^{UV}$ are the relations of the $(e_{\cS}+\ell_{\cS})$-class
association scheme $\fuse{\cS}{\sym(\cF_{\cS})}$. It is clear
that this scheme is imprimitive with $\cF=\cF_{\cS}$ and
$\cI=\cI_{\cS}$, and that its quotient scheme is trivial.
Because \eqref{coh_conf_prod} follows from
\eqref{coco_uniform}, this scheme is uniform. We have thus
shown a one-to-one correspondence between uniform association
schemes and uniform coherent configurations.

\begin{proposition}\label{one-one} If $(X,\cR)$ is a uniform association
scheme, then $(X,\{R_i^{UV}| i \in \cI(U,V); U,V \in \cF\})$ is
a uniform coherent configuration. Conversely, if $(X,\cS)$ is a
uniform coherent configuration, then after relabeling
$(X,\{\cup_{U,V} S_i^{UV}| i = 0,\dots,e_{\cS}+\ell_{\cS}\})$ is
a uniform association scheme on $X$.
\end{proposition}

\noindent We will now use this one-to-one correspondence to show
that every dismantlable association scheme is uniform.

\begin{theorem}
\label{dismunif}
An association scheme is dismantlable if and only if it is uniform.
\end{theorem}

\begin{proof}
One direction has already been shown in Theorem \ref{uniformdismantle}.

Let $(X,\cR)$ be a dismantlable association scheme. Because
bipartite schemes are uniform, we may assume that $w \geq 3$.
We must first check that the quotient scheme is trivial. To see
this, it suffices to show that, for any three distinct fibres
$U$, $V$ and $W$, $\cI(U,V)=\cI(V,W)$. But this is clear since
the dismantled scheme on vertex set $Y=U \cup V \cup W$ is
still imprimitive and there is only one choice for its
quotient: the trivial scheme on three vertices. So $\cI(U,V) =
\cI(V,W)=\overline{\cI}$.

Next, we claim that $\cS:=\{R_i^{UV}| i \in \cI(U,V); U,V \in
\cF\}$ forms a coherent configuration on $X$.
In order to do this, we will have
to consider the intersection numbers of the dismantled schemes
$(Y,\cR^Y)$ where $Y$ is a union of fibres, which we denote
by $p_{ij}^h(Y)$.
To establish the
claim, we first observe that the non-empty relations $R_i^{UV}$
form a partition of $X \times X$ and that
$(R_i^{UV})^\top=R_i^{VU}$.

Now pick an arbitrary triple of
relations $R_i^{UV},R_j^{VW},R_h^{UW},$ with $i\in \cI(U,V)
,j\in \cI(V,W),h\in \cI(U,W)$. We have to show that the number
$$
\lm_{ijh}^{UVW}(u,x):= |R_i^{UV}(u)\cap (R_j^{VW})^{\top}(x)| = |R_i(u)\cap R_j(x)\cap V|
$$
does not depend on the pair $(u,x)\in R_h^{UW}$.

If $U=V$ then $i\in \cI$, implying $R_i(u)\cap V = R_i(u)$.
Therefore $ \lm_{ijh}^{UVW}(u,x) = |R_i(u)\cap
R_j(x)|=p_{ij}^h.$ Analogously, $ \lm_{ijh}^{UVW}(u,x)
=p_{ij}^h $ if $V=W$.

Next, we consider the case that $U\neq V$ and $U=W$. In this
case $h\in\cI$, while $i,j\in\ovr{\cI}$. Consider the scheme
$(Y,\cR^Y)$, where $Y=U\cup V$.  For $u,x \in U=W$, we have
\begin{equation}\label{pijhY}
p_{ij}^h(Y)=|R_i^{Y}(u)\cap (R_j^{Y})^\top(x)| = |R_i(u)\cap R_j(x)\cap Y|.
\end{equation}
Because $i,j\in\ovr{\cI}$, the intersections $R_i(u)\cap U$ and
$R_j(x)\cap U$ are empty. Therefore $R_i(u)\cap R_j(x)\cap Y =
R_i(u)\cap R_j(x)\cap V$, and hence $\lm_{ijh}^{UVW}(u,x) =
p_{ij}^h(U\cup V).$

The last case is the one in which the fibres $U,V,W$ are
pairwise distinct. Then $i,j,h\in\ovr{\cI}$. Consider the
scheme $(Y,\cR^Y)$, where $Y=U\cup V\cup W$. As before, we have
\eqref{pijhY} for $u \in U, x \in W$. Because
$i,j\in\ovr{\cI}$, we obtain $R_i(u)\cap Y \subseteq V\cup W$
and $R_j(x)\cap Y\subseteq U \cup V$. This implies that
$R_i(u)\cap R_j(x)\cap Y = R_i(u)\cap R_j(x)\cap V$, and
therefore $\lm_{ijh}^{UVW}(u,x) = p_{ij}^h(U\cup V\cup W)$.

Thus we proved that the relations in $\cS$ form a coherent
configuration, with intersection numbers
\begin{equation}\label{str_const}
p_{R^{UV}_i R^{VW}_j}^{R^{UW}_h}=\lm_{ijh}^{UVW}=
\begin{cases}
p_{ij}^h & \mbox{ if }U=V \mbox{ or } V=W;\\
p_{ij}^h(U\cup V) & \mbox{ if }U\neq V, U=W;\\
p_{ij}^h(U\cup V\cup W) & \mbox{ if }U\neq V,V\neq W,W\neq U.
\end{cases}
\end{equation}

Finally, we shall show that the coherent configuration is
uniform. By the above one-to-one correspondence between uniform
association schemes and uniform coherent configurations this
proves the theorem. To show that the configuration is uniform,
we have to prove that
$\displaystyle{\lm_{ijh}^{UVW} } = \displaystyle{\lm_{ijh}^{U'V'W'}}$
whenever the triples $(U,V,W)$ and $(U',V',W')$ have the same type.

For $U=V$ (and, therefore, $U'=V'$), or if $V=W$, this is clear.
If $U\neq V, U=W$ and $U'\neq V',U'=W'$, then
$i,j\in\ovr{\cI},h\in\cI$. In this case
we have to show that $p_{ij}^h(U\cup V) = p_{ij}^h(U'\cup V')$. To prove this,
it is sufficient to show that $p_{ij}^h(U\cup V) =
p_{ij}^h(V\cup W)$ holds for any triple $(U,V,W)$ of pairwise
distinct fibres. So, consider a scheme $(Y,\cR^Y)$, where $Y=U\cup
V\cup W$. Because $h\in\cI$, $R_h^{Y} = R_h^{U}\cup
R_h^{V}\cup R_h^{W}$. Pick an arbitrary pair $(u,u')\in R_h^{U}$,
that is, $(u,u')\in R_h$ and $u,u'\in U$. Because $i,j\in\ovr{\cI}$, we have that
$$
p_{ij}^h(Y)=|R_i(u)\cap R_j(u')\cap Y| = |R_i(u)\cap R_j(u')\cap V| + |R_i(u)\cap R_j(u')\cap W|=
$$
$$
|R_i(u)\cap R_j(u')\cap (U\cup V)| + |R_i(u)\cap R_j(u')\cap(U\cup W)| =
p_{ij}^h(U\cup V)+p_{ij}^h(U\cup W).
$$
The same argument with $(x,x')\in R_h^{W}$ shows that $p_{ij}^h(W\cup V)+p_{ij}^h(W\cup U) = p_{ij}^h(Y)$, and hence
$p_{ij}^h(U\cup V) = p_{ij}^h(V\cup W)$.

Consider now the remaining case where the triples $(U,V,W)$ and
$(U',V',W')$ consist of pairwise distinct fibres. In this case
$i,j,h\in\ovr{\cI}$ and we have to show that $p_{ij}^h(U\cup V
\cup W) = p_{ij}^h(U'\cup V' \cup W')$. If $w = 3$, then there
is nothing to prove, so we may assume that $w\geq 4$. In this
case it is sufficient to show that $p_{ij}^h(U\cup V \cup W) =
p_{ij}^h(V\cup W \cup Z)$ holds for each quadruple $U,V,W,Z$ of
pairwise distinct fibres. The arguments for this are similar as
in the previous case. Consider the scheme $(Y,\cR^Y)$, where
$Y=U\cup V\cup W\cup Z$. Then it follows from considering pairs
$(u,y)\in R_h^{UV}$ and $(y,z)\in R_h^{VZ}$ that
$$p_{ij}^h(U\cup V \cup W)+ p_{ij}^h(U\cup V \cup Z) =p_{ij}^h(Y)=
p_{ij}^h(V\cup Z \cup W)+p_{ij}^h(V\cup Z \cup U),$$ which
finishes the proof.
\end{proof}

\noindent As an immediate consequence,
we obtain important structural information about dismantled schemes.

\begin{corollary}
\label{Cdismantledallsame}
Let $(X,\cR)$ be a dismantlable association scheme with $w$ fibres.
If $2\le w' \le w$ and each of $Y,Y' \subseteq X$ are expressible as
a union of $w'$ fibres, then the dismantled schemes $(Y,\cR^Y)$ and
$(Y',\cR^{Y'})$ have the same parameters (i.e., same eigenmatrices $P$
and $Q$ and same intersection numbers and Krein parameters, with appropriate
orderings of their relations and idempotents).
\end{corollary}

\begin{proof}
It follows from Definition \ref{Dunifscheme} that the
parameters of the dismantled scheme $(Y,\cR^Y)$  depend only on
$w'$ and the parameters $a_{ij}^h$ and not on the choice of $Y$
itself.
\end{proof}

\subsection{Q-Higman schemes}
\label{Subsec:QHigman}

In the previous section, we have seen that uniformity of a scheme is
equivalent to dismantlability. In this section, we give a
characterization of uniform schemes in terms of the Krein
parameters (through so-called Q-Higman schemes) and study
the idempotents of uniform schemes.

\subsubsection{Krein parameters of Q-Higman schemes}\label{uniformschemes}

With cometric Q-antipodal association schemes in mind, we
consider an imprimitive association scheme with $\cJ=\{0,d\}$,
$\cJ_j=\{j,d-j\}$ for $j=0,1,\dots,\ell-1<\frac{d}{2}$, and
$\cJ_j=\{j\}$ for $j=\ell,\dots,d-\ell$ (for some $\ell$).

For such a scheme we consider the dual intersection matrix
$L^*_d$ with entries $(L^*_d)_{ij}:=q^i_{dj}$. First note that
$\rho_j=1+q^j_{dj}$. If $j < \ell$ or $j
> d-\ell$, then from $E_j \circ (E_0+E_d) = \frac{1}{n} \pi(E_j)
=\frac{1}{v}(1+q^j_{dj})(E_j+E_{d-j})$, we find that $E_j \circ E_d
= \frac{1}{v}(q^j_{dj}E_j + (1+q^j_{dj})E_{d-j})$, and hence that
$q^{d-j}_{dj}=1+q^j_{dj}$, and $q^i_{dj}=0$ for $i \neq j,d-j$.

For $\ell \leq j \leq d-\ell$, we find from $E_j \circ
(E_0+E_d) = \frac{1}{v}(1+q^j_{dj})E_j$ that $q^i_{dj}=0$ for
$i \neq j$, and hence $q^j_{dj}=w-1$. In other words, the only
nonzero entries of $L^*_d$ are on the diagonal and the
antidiagonal.

For $j < \ell$ or $j > d-\ell$, we may combine the facts $q^{d-j}_{dj}=1+q^j_{dj}$ and
$q^{j}_{dj}+q^{j}_{d,d-j}=m_d=w-1$ to find $(1+q^j_{dj})m_{d-j}=(w-1-q^j_{dj})m_j$. This implies that $m_{d-j} \leq
(w-1)m_j$ with equality if and only if $q^j_{dj}=0$. We thus obtain the following:

\begin{lemma}
\label{cosetssize2krein}
Consider an imprimitive association scheme with
$\cJ=\{0,d\}$, $\cJ_j=\{j,d-j\}$ for
$j=0,1,\dots,\ell-1<\frac{d}{2}$, and $\cJ_j=\{j\}$ for
$j=\ell,\dots,d-\ell$. If $\ell \leq j \leq d-\ell$, then
$q^i_{dj}=0$ for $i \neq j$ and $\rho_j =q^j_{dj}+1=w$. If $j <
\ell$ or $j > d-\ell$, then $\rho_j=q^{d-j}_{dj}=1+q^j_{dj}$
and $q^i_{dj}=0$ for $i \neq j,d-j$, and moreover, $m_{d-j}
\leq (w-1)m_j$ with equality if and only if $q^j_{dj}=0$.
\end{lemma}

\noindent The case of equality is one of the motivations for
the following definition.

\begin{definition}
\label{def:Qanti}
An imprimitive association scheme is called Q-Higman
if for some $\ell$ such that $1 \leq \ell < \frac{d}{2}+1$ and for
some ordering of the primitive idempotents, we have that
$\cJ=\{0,d\}$, $\cJ_j=\{j,d-j\}$ for $j=0,1,\dots,\ell-1$,
$\cJ_j=\{j\}$ for $j=\ell,\dots,d-\ell$, and $q^d_{jj}=0$ (or
equivalently $m_{d-j}=(w-1)m_j$) for $j=0,1,\dots,\ell-1$.
\end{definition}

\noindent It is important to note that this Q-Higman property
is formulated entirely in terms of
the Krein parameters, in particular in terms of the dual
intersection matrix $L^*_d$.

\begin{proposition}\label{Qantikrein} An association scheme is Q-Higman if
and only if for some $\ell$ such that $1 \leq \ell < \frac{d}{2}+1$,
for some $w$, and some ordering of the idempotents it holds that
$q^{j}_{d,d-j}=w-1$ for $j < \ell$, $q^{j}_{d,d-j}=1$ and
$q^{j}_{dj}=w-2$ for $j > d-\ell$, $q^{j}_{dj}=w-1$ for $\ell \leq j
\leq d-\ell$, and $q^i_{dj}=0$ for all other values of $i$ and $j$.
Moreover, if this is the case, then $\rho_j=1$ if $j < \ell$,
$\rho_{j}=w$ if $\ell \leq j \leq d-\ell$, and $\rho_j=w-1$ if $j >
d-\ell$.
\end{proposition}

\begin{proof} If the scheme is Q-Higman, then the stated
properties follow from the above considerations. On the other
hand, suppose that these properties hold. Then it follows that
$v(E_0+E_d) \circ (E_0+E_d)=w(E_0+E_d)$ and that $\langle
E_0,E_d \rangle$ is a $\circ$-subalgebra. This means that the
scheme is imprimitive with $\cJ=\{0,d\}$ and fibres of size
$\frac{v}{w}$. The equivalence classes of $\sim^*$ then easily
follow, and so does the conclusion that the scheme is
Q-Higman.
\end{proof}

\noindent We note that the standard relations between the Krein
parameters of a scheme (e.g., see \cite[Lemma 2.3.1]{bcn}) give some
more specific information on those of Q-Higman schemes. It can
for example be derived (from \cite[Lemma 2.3.1]{bcn} or directly by
working out the product $E_i \circ E_j \circ E_d$ in different ways)
that if $j<\ell$ and $i$ is arbitrary, then
$q^h_{i,d-j}=(w-1)q^{d-h}_{ij}$ for $h < \ell$,
$q^h_{i,d-j}=(w-1)q^{h}_{ij}$ for $\ell \leq h \leq d-\ell$, and
$q^h_{i,d-j}=q^{d-h}_{ij}+(w-2)q^h_{ij}$ for $h > d-\ell$. It also
follows that $q^h_{ij}=q^{d-h}_{ij}$ for all $i$, $\ell \leq j\leq
d-\ell$ and $h < \ell$. In the cometric Q-antipodal case, we include
these observations in Lemma \ref{Lqijk2} below.

\subsubsection{The idempotents of uniform schemes}

In this section we shall show one of our main results, i.e., that
Q-Higman schemes and uniform schemes are the same. For this we
will again use the correspondence to uniform coherent
configurations.

We remind the reader that $\cA=\langle A_i\,|\,i
=0,\dots,d\rangle$ is the Bose-Mesner algebra of the association
scheme under consideration, and that $\cB =\langle A_i\,|\,i\in
\cI\rangle$ is the Bose-Mesner subalgebra on the fibres.
Moreover, we let
$$\cD:=\langle A_i\,|\,i\not\in \cI\rangle.$$
In order to show that a uniform scheme is Q-Higman, and to find
relations with its dismantled schemes, we study its idempotents. We
start off with the case of bipartite schemes, i.e., imprimitive
schemes with two fibres.

\begin{lemma}\label{bipartite} A bipartite scheme is
Q-Higman. Each primitive idempotent of $\cB$ that is not a
primitive idempotent of $\cA$ is of the form $E+E'$, where $E$
and $E'$ are primitive idempotents of $\cA$, and $E-E' \in
\cD$.
\end{lemma}

\begin{proof}
Consider a bipartite scheme with fibres $U$ and $V$. Because all relations $R_i, i \notin \cI$ are bipartite, it
follows that $E=\mtrx{E_{UU}}{E_{UV}}{E_{VU}}{E_{VV}}$ is a primitive idempotent if and only if
$E'=\mtrx{E_{UU}}{-E_{UV}}{-E_{VU}}{E_{VV}}$ is a primitive idempotent. Moreover, if $E$ and $E'$ are indeed primitive
idempotents of $\cA$ and $E_{UV} \neq 0$, or equivalently,
$E \notin \cB$, then $E+E'$ is a primitive idempotent of the Bose-Mesner subalgebra $\cB$, and $E-E'\in \cD$. This
implies that the primitive idempotents of $\cB$ that are not primitive idempotents of $\cA$ are of the form $E+E'$,
where $E$ and $E'$ are primitive idempotents of $\cA$, and $E-E'\in \cD$. Thus, all sets $\cJ_j$ have size at most two.
Moreover, the multiplicities of the idempotents $E$ and $E'$ are equal, because $\text{trace}(E)=\text{trace}(E')$.
Thus, the scheme is Q-Higman.
\end{proof}

\begin{lemma}\label{p:matrixD} Consider a uniform association scheme.
Let $F\in\cB$ be a primitive idempotent of $\cB$. Then $F$ is a
primitive idempotent of $\cA$ if and only if $F\cD = \{0\}$.
Let $Y$ be a union of at least two fibres. Then $F^Y$ is a
primitive idempotent of $\cB^Y$. Moreover, $F^Y$ is a primitive
idempotent of $\cA^Y$ if and only if $F$ is a primitive
idempotent of $\cA$.
\end{lemma}

\begin{proof}
An idempotent $F$ of $\cA$ is primitive if and only if $FA$ is proportional to
$F$ for each $A\in\cA$. Because $F$ is a primitive idempotent of $\cB$, $FA$ is
proportional to $F$ for each $A\in\cB$. Therefore $F$ is a primitive idempotent
of $\cA$ if and only if $FA$ is proportional to $F$ for each $A\in \cD$. So
consider $A\in \cD$. Because $F$ is block-diagonal and $A^{U}=0$ for $U\in\cF$,
we obtain $(FA)^{U}=0$. Therefore $FA$ is proportional to $F$ if and only if
$FA=0$.

Because of the block-diagonal structure of $\cB$, $F^Y$ is
clearly a primitive idempotent of $\cB^Y$, and $(F\cD)^Y =
F^Y\cD^Y$. Because the linear map $A\mapsto A^Y$ is a bijection
between $\cA$ and $\cA^Y$, it follows that $F\cD= \{0\}$ if and
only if $F^Y\cD^Y= \{0\},$ hereby proving the final statement
of the lemma.
\end{proof}

\begin{theorem}\label{uniformidempotents}
Consider a uniform association scheme. Let $F_0,\dots,F_e\in\cB$
be a complete set of primitive idempotents of $\cB$, ordered
such that $F_0,\dots,F_{\ell-1}$ are not primitive in $\cA$, and
$F_{\ell},\dots,F_e$ are primitive in $\cA$. Then for each
$j=0,\dots,\ell-1$ there exists a matrix $D_j\in\cD$ such that
for each union $Y$ of $w' \geq 2$ fibres, the matrices
$\frac{1}{w'}(F_j^Y+ D_j^Y)$ and $F_j^Y-\frac{1}{w'}(F_j^Y+
D_j^Y), j=0,\dots,\ell-1$, and $F_{\ell}^Y,\dots,F_e^Y$ are the
primitive idempotents of $\cA^Y$.
\end{theorem}
\begin{proof}
First of all it follows from Lemma \ref{p:matrixD} that the
matrices $F_{\ell}^Y,\dots,F_e^Y$ are primitive idempotents of
$\cA^Y$. Secondly, we fix $j \in \{0,\dots,\ell-1\}$ for the
moment, and let $F:=F_j$. We then claim that there is a matrix
$D\in\cD$, which is unique up to sign, such that for any two
distinct fibres $U, V$, we have
\begin{align}
F^{UU} D^{UV}&=D^{UV} F^{VV} = D^{UV}, \notag \\
D^{UV}D^{VU}&=F^{UU}, \label{matrixD} \\
D^{VU}D^{UV}&=F^{VV}. \notag
\end{align}
To prove this claim, we first fix two fibres $U$ and $V$, let
$Z:=U\cup V$, and consider the bipartite dismantled scheme on
$Z$. By Lemma \ref{p:matrixD} we have that $F^Z$ is a primitive
idempotent of $\cB^Z$ which is not a primitive idempotent of
$\cA^Z$. From Lemma \ref{bipartite} we obtain that $F^Z = E +
E'$, where $E$ and $E'$ are primitive idempotents of $\cA^Z$
such that $E-E' \in \cD^Z$. Because the map $D\mapsto D^Z$ is a
bijection between $\cD$ and $\cD^Z$, there is a matrix $D \in
\cD$ such that $D^Z=E-E'$. Because $E$ and $E'$ are orthogonal,
this matrix $D$ satisfies $F^ZD^Z=D^ZF^Z=D^Z$ and
$(D^Z)^2=F^Z$. It then follows that $D$ satisfies
\eqref{matrixD} for the fixed fibres $U$ and $V$. Now we use the
fact that $\sym(\cF)$ acts doubly transitively on $\cF$:
by applying algebraic automorphisms $\sigma \in\sym(\cF)$
to these equations, we find that
they hold for all fibres $U,V$.

It remains to prove uniqueness of $D$. Let $M\in\cD$ be a
matrix satisfying~\eqref{matrixD}, i.e., $F^ZM^Z=M^ZF^Z=M^Z$
and $(M^Z)^2=F^Z$. Because $F^Z=E+E'$ and $E,E'$ are primitive
idempotents of $\cA^Z$, there exist four solutions of the
equation $(M^Z)^2=F^Z$ with $M^Z\in\cA^Z$, namely $\pm E \pm
E'$ (this easily follows by writing $M^Z$ as a linear
combination of primitive idempotents of $\cA^Z$). On the other
hand, the matrices $\pm F^Z,\pm D^Z$ satisfy this equation.
Therefore $M^Z=\pm D^Z$. Again, because the map $D\mapsto D^Z$
is a bijection between $\cD$ and $\cD^Z$, we obtain that $M=\pm
D$, and the claim is proven.

The above considerations show the existence of $D\in\cD$ such that
$FD=D$ and $D^2 = (w-1)F+(w-2)D$ for the case $w=2$. Now let us
assume that $w \geq 3$. Fix three arbitrary but distinct fibres, say
$U,V,W$, and consider the product $D^{UV}D^{VW}$. Because of
uniformity this product belongs to $\cA^{UW}$. Therefore there
exists a $G\in\cD$ such that $G^{UW} = D^{UV}D^{VW}$. It follows
from~\eqref{matrixD} that $F^{UU} G^{UW} = G^{UW} F^{WW} = G^{UW}$,
$G^{UW}G^{WU}=F^{UU}$, and $G^{WU}G^{UW}=F^{WW}$. From the above
claim it then follows that $G=\veps D$, where $\veps=\pm 1$. Thus
$D^{UV} D^{VW} =\veps D^{UW}$, and after replacing $D$ by $\veps D$
this becomes $D^{UV} D^{VW} =D^{UW}$. Applying --- as before ---
algebraic automorphisms $\sigma \in\sym(\cF)$ to this equality we
obtain that $D^{U'V'} D^{V'W'} = D^{U'W'}$ for any triple of
pairwise distinct fibres $U',V',W'$.

If $Y$ is a union of $w' \geq 2$ fibres, then a routine
calculation shows that $(D^Y)^2 = (w'-1)F^Y+(w'-2)D^Y$. After
releasing the fixation of $j$ by indexing $F$ and $D$, we thus
obtain that
\begin{equation}\label{eq:ED}
F_j^YD_j^Y=D_j^Y \text{~and~} (D_j^Y)^2 = (w'-1)F^Y_j+(w'-2) D^Y_j.
\end{equation}
For fixed $Y$, it remains to show that the matrices
$E_j:=\frac{1}{w'}(F_j^Y+ D_j^Y)$ and
$E_j':=F_j^Y-\frac{1}{w'}(F_j^Y+ D_j^Y), j=0,\dots,\ell-1$, and
$F_{\ell}^Y,\dots,F_e^Y$ are the primitive idempotents of
$\cA^Y$. It follows from~\eqref{eq:ED} that $E_j,E_j'$ are
pairwise orthogonal idempotents. To show that $E_j, E_j'$ are
orthogonal to $E_h,E_h'$ for $h\neq j$, and to $F_h^Y$ for $h
\geq \ell$, it is sufficient to check that $F^Y_j D^Y_h = F^Y_h
D^Y_j = D^Y_j D^Y_h =0$. These equations hold because $F^Y_h
D^Y_j = F^Y_h F^Y_j D^Y_j =0$ and $D^Y_j D^Y_h = F^Y_j D^Y_j
F^Y_h D^Y_h = F^Y_j F^Y_h D^Y_j D^Y_h = 0$.

Thus we have $2\ell+e+1-\ell=e+1+\ell$ pairwise orthogonal
idempotents of $\cA^Y$. It remains to show that $d+1=e+1+\ell$.
Because $d,e,\ell$ do not depend on $w'$ (for $w'\geq 2$; for
$\ell$ this follows from Lemma \ref{p:matrixD}), it is enough
to check this equality for $w'=2$. But in the case of $w'=2$
each primitive idempotent of $\cB$ is either primitive in $\cA$
or splits into a sum of two primitive idempotents of $\cA$, as
we saw in Lemma \ref{bipartite}. This implies that
$d+1=e+1+\ell$.
\end{proof}

\begin{corollary}\label{coruniformQ} A uniform association scheme is Q-Higman.
\end{corollary}

\begin{proof} Consider a uniform association scheme. Apply
Theorem \ref{uniformidempotents} with $Y=X$ and $w'=w$
to see that  the sets $\cJ_j$ have size at most two, i.e.,
its primitive idempotents that are not primitive idempotents of
$\cB$ come in pairs $E_j, E_j'$. The corresponding
multiplicities satisfy
$m_j'=\text{trace}(E_j')=\frac{w-1}{w}\text{trace}(F_j)=(w-1)\text{trace}(E_j)=(w-1)m_j$,
which concludes the proof.
\end{proof}

\noindent Note that implicitly we have shown that $\ell \leq e + 1$ in a uniform association scheme (because this is
equivalent to $\ell<\frac d2+1$). In other words, the number of relations between two fibres is at most the number of
relations within each fibre. This also follows from a more general result for coherent configurations with commutative
schemes on the fibres. Indeed, according to Higman \cite[p.\ 227]{HigmanCA} and Weisfeiler \cite[Cor.~14, p.\
87]{weis}, the number of relations between two fibres is at most the number of relations in each of these two fibres.
The special case where the number of relations between two fibres and within each fibre is the same (for all fibres)
comprises the so-called balanced coherent configurations, and these have been studied by Hirasaka and Sharafdini
\cite{HR}.

The next result also follows easily from Theorem \ref{uniformidempotents}.

\begin{corollary}\label{dismantledidempotents} Consider a uniform association scheme, with primitive idempotents
$E_j, j=0,\dots,d$ (ordered as in Definition \ref{def:Qanti}),
and let $Y$ be a union of $w' \geq 2$ fibres. Then the
primitive idempotents of the dismantled scheme on $Y$ are
$\overline{E}_j:=\frac{w}{w'}E^Y_j$ and
$\overline{E}_{d-j}:=E^Y_{d-j}+E^Y_j- \frac{w}{w'}E^Y_j$,
$j=0,\dots,\ell-1$, and $\overline{E}_j:=E_j^Y$,
$j=\ell,\dots,d-\ell$.
\end{corollary}

\noindent To show the converse of Corollary \ref{coruniformQ}, i.e., that a Q-Higman scheme is uniform, we use the
following lemma, whose proof is similar to the dismantlability proof of a cometric Q-antipodal scheme in
\cite[Thm.~4.7]{mmw}.

\begin{lemma}\label{p_1}
Consider a Q-Higman scheme. Then for each fibre $U$,
$$
E_j I^U E_h =
\begin{cases}
w^{-1} E_j & {\rm if\ }  h = j {\rm\ and\ } j =0,\dots,\ell-1;\\
E_j I^U - w^{-1} E_j & {\rm if\ }  h = d-j {\rm\ and\ } j =0,\dots,\ell-1;\\
I^U E_{d-j}  - w^{-1} E_{d-j} & {\rm if\ }  h = d-j {\rm\ and\ } j=d-\ell+1,\dots,d;\\
E_j I^U - I^U E_{d-j}  + w^{-1} E_{d-j} & {\rm if\ }  h = j {\rm\ and\ } j=d-\ell+1,\dots,d;\\
E_j I^U& {\rm if\ }  h =j {\rm\ and\ } j =\ell,\dots,d-\ell;\\
0           & {\rm otherwise.}
\end{cases}
$$
\end{lemma}
\begin{proof} Similar as in the proof of \cite[Thm.~4.7]{mmw}, it follows from \cite[p.\ 61, Eq. 9]{bcn} that
\begin{equation}\label{eq_main}
\parallel v E_j I^U E_h - n\delta_{jh} E_j\parallel^2 = q_{jh}^dn^2(w-1).
\end{equation}
To start with the bottom line of the expression for $E_j I^U E_h$:
if $h\nsim^* j$ then $h \neq j$ and $q_{jh}^d=0$, and we obtain from
\eqref{eq_main} that $E_j I^U E_h = 0$.

If $h = j$ with $j =0,\dots,\ell-1$, then $q_{jh}^d =0$ and so
$E_j I^U E_j = w^{-1} E_j$.

If $h = d-j$ with $j =0,\dots,\ell-1$, then
$$
E_jI^U E_{d-j} = E_j I^U (I-\sum_{i\neq d-j} E_i) = E_jI^U - E_jI^U
E_{j} = E_j I^U - w^{-1} E_j.
$$

For $j=d-\ell+1,\dots,d$, we have that $0\le d-j \le \ell-1$, hence
from the above it follows that $E_{d-j}I^U E_j = E_{d-j} I^U -
w^{-1} E_{d-j}$. By transposing this expression we obtain that
$E_jI^U E_{d-j} = I^U E_{d-j} - w^{-1} E_{d-j}$.

Also for $j=d-\ell+1,\dots,d$ we have that
$$
E_jI^U E_j = E_j I^U (I-\sum_{i\neq j} E_i) = E_j I^U - E_j I^U
E_{d-j} = E_j I^U - I^U E_{d-j}  + w^{-1} E_{d-j}.
$$
For $j =\ell,\dots,d-\ell$, the idempotent $E_j$ is
block-diagonal, implying that $E_jI^U E_j = E_j I^U$.
\end{proof}

\begin{theorem}\label{p_coco} Consider a Q-Higman association scheme. Then
$$\mathfrak M:=\left\langle E_j^{UV}\,|\, j =0,\dots,d-\ell \mbox{ and }U,V\in \cF\right\rangle$$
is a coherent algebra corresponding to a uniform coherent
configuration.
\end{theorem}
\begin{proof}
We shall show that $\mathfrak M$ is closed with respect to
transposition, ordinary matrix multiplication, and entrywise
multiplication, and contains $I$ and $J$, thus proving it is a
coherent algebra.

First however, we claim that $E_j^{UV} \in \mathfrak M$ also
for $j=d-\ell+1,\dots,d$. Indeed, in this case $0 \le d-j \le \ell-1$
and $v^{-1}E_j = E_d\circ E_{d-j}$. Therefore
\begin{align*}
E_{j}^{UV} =
v E_d^{UV} \circ E_{d-j}^{UV}
&=
\begin{cases}
-J^{UV} \circ E_{d-j}^{UV} & \mbox{ if } U\neq V \\
(w-1) J^{UV} \circ E_{d-j}^{UV} & \mbox{ if } U=V
\end{cases} \\
&=
\begin{cases}
-E_{d-j}^{UV} & \mbox{ if } U\neq V \\
(w-1) E_{d-j}^{UV} & \mbox{ if } U=V
\end{cases}
\in \mathfrak M.
\end{align*}
Hence $E_j^{UV}\in \mathfrak M$ for each $j,U,V$. This implies
that  $E_j \in \mathfrak M$ for each $j$, and hence $I,J\in
\mathfrak M$.

Concerning the closure properties, note that closure with
respect to transposition is evident. Closure with respect to
matrix multiplication follows from Lemma \ref{p_1}, because it
implies that
\begin{equation}\label{eq_product}
E_i^{UV}E_j^{WZ} =\delta_{VW}\,\delta_{ij}\,
\lambda \,E_i^{UZ}\in \mathfrak M,
\end{equation}
where $\lambda = w^{-1}$ for $i=0,\dots,\ell-1$ and
$\lambda=\delta_{WZ}$ for $i=\ell,\dots,d-\ell$ (here $\delta$ is
the Kronecker delta). Closure with respect to entrywise
multiplication follows from
$$
E_j^{UV}\circ E_h^{UV} =
(E_j \circ E_h)^{UV} = v^{-1}\sum_{i=0}^d q_{jh}^i E_i^{UV} \in
\mathfrak M.$$

It remains to show uniformity. Note that it is clear from the
above that $\mathfrak M$ contains all the matrices
$A_i^{UV}$; the nonzero matrices among these  form a basis of Schur
idempotents for the
corresponding coherent configuration. Because $A_i^{UV}$ can be
expressed as a linear combination of the $E_j^{UV},
j=0,\dots,d-\ell$, it follows from \eqref{eq_product} that the
coherent configuration is uniform.
\end{proof}

\begin{corollary}\label{Qantipodaluniform} A Q-Higman scheme is uniform. Any
dismantled scheme of such a scheme is also Q-Higman.
\end{corollary}

\begin{proof} The first statement follows from Theorem \ref{p_coco} and the correspondence between uniform
coherent configurations and uniform schemes (Proposition
\ref{one-one}). The second statement follows from
dismantlability (Proposition \ref{uniformdismantle}) and the
converse of the first part (Corollary \ref{coruniformQ}).
\end{proof}

\noindent We thus have proven the following.

\begin{theorem}
\label{TuniformiffQHigman}
An association scheme is uniform if and only if
it is Q-Higman.
\end{theorem}

\section{Cometric Q-antipodal schemes}\label{sec:cometricQ}

A cometric association scheme (with a Q-polynomial ordering
$E_0, E_1,\dots, E_d$) is called Q-antipodal if it is imprimitive
with $\cJ=\{0,d\}$. It is called Q-bipartite if it is
imprimitive with $\cJ=\{0,2,4,\dots\}$, or equivalently if
$a_i^*=0$ for all $i$, cf. \cite{suzimprim}.

It was shown by Suzuki \cite{suzimprim} that an imprimitive cometric $d$-class
association scheme is Q-antipodal, Q-bipartite, or both, unless possibly when
$d=4$ or $d=6$. The exceptional cases for $d=4$ and $d=6$ were
later ruled out by Cerzo and Suzuki \cite{cerzo} and Tanaka and Tanaka
\cite{TT2011EJC}, respectively. Here we will consider the Q-antipodal case.

\subsection{Uniformity}

Consider a cometric Q-antipodal association scheme. In this case, it
follows that the equivalence classes of the relation $\sim^*$ are
$\cJ_j=\{j,d-j\}, j=0,1,\dots,\lfloor \frac{d}{2} \rfloor$. So the
primitive idempotents of the Bose-Mesner subalgebra $\cB$ are
$F_j=E_j+E_{d-j}, j< \frac{d}{2}$, and
$F_{\frac{d}{2}}=E_{\frac{d}{2}}$ for $d$ even. Note also that
$q^j_{dj}=0$ for $j<\frac{d}{2}$, hence a cometric Q-antipodal
scheme is Q-Higman (with $\ell=\lceil \frac{d}{2} \rceil$), and
therefore it is also uniform, and dismantlable. On the other hand,
we will show now that a uniform cometric scheme is Q-antipodal.

\begin{theorem}\label{cometricuniformQantipodal} A cometric association scheme is
uniform if and only if it is Q-antipodal.
\end{theorem}

\begin{proof} One direction is clear from the above. Consider
now a cometric scheme that is uniform with imprimitivity system
$\cF$. So the scheme is Q-Higman, and let us assume that the
idempotents are ordered as in Definition \ref{def:Qanti}; in
particular we have $\cJ=\{0,d\}$. In order to show that the
scheme is cometric Q-antipodal, it suffices to show that $E_d$
is last in a Q-polynomial ordering too. In the case $d=3$,
however, a somewhat degenerate case also arises where $E_d$ is
second in the Q-polynomial ordering, but in this ordering $E_1$
is last and there is a second imprimitivity system $\cF'$ with
subscheme corresponding to $\cJ' = \{0,1\}$.

We first note that it is clear that $E_d$ cannot be a Q-polynomial generator,
and that this proves the case $d=2$.

Next, consider the case $d > 3$. Then $E_d$ must take the last
position in any Q-polynomial ordering as $E_i \circ E_d \in \langle
E_i, E_{d-i} \rangle$ eliminates positions from three up to $d-1$
(taking $E_i$ to be the Q-polynomial generator) and position two
(taking $i=d$ and some $E_j$, $j\in \{1,2,\dots,d-1\}$, in position
four).

For the case $d=3$, we apply several properties of the Krein
parameters from Proposition \ref{Qantikrein}. Consider a
Q-polynomial ordering, and assume that $E_3$ is not in its last
position. Because $q^3_{11}=0$, this ordering cannot be
$E_0,E_1,E_3,E_2$, hence it must be $E_0,E_2,E_3,E_1$. In this
latter case, the scheme is cometric Q-bipartite, hence $q^i_{2i}=0$
for all $i$. Because $q^2_{32}=w-2$, it follows that $m_2=\sum_j
q^2_{2j}=w-1$, which in turn shows that $m_1=1$. Thus $\{E_0,E_1\}$
induces another imprimitivity system $\cF'$ with $\cJ'=\{0,1\}$.
Because $E_1$ is last in the Q-polynomial ordering under
consideration, this implies that also in this case the scheme is
cometric Q-antipodal.
\end{proof}

\noindent An interesting consequence of Theorem
\ref{cometricuniformQantipodal} is that among the cometric
association schemes, the Q-antipodal ones can be recognized
combinatorially.

The exceptional case in the above proof is realized only by the
rectangular scheme $R(w,2), w>2$ (the direct product of two trivial
schemes; on $w$ and $2$ vertices). Note that this cometric
Q-antipodal Q-bipartite scheme has one Q-polynomial ordering, but
two ``uniform" imprimitivity systems; for one such system there is a
uniform ordering of the idempotents (as in Definition
\ref{def:Qanti}) that matches the Q-polynomial ordering, for the
other not. The proof of Theorem \ref{cometricuniformQantipodal} thus
implies the following.

\begin{corollary}\label{orderfree}
Consider a uniform $d$-class association scheme with $\cJ=\{0,d\}$.
If the scheme is cometric then $E_d$ is in the last position in any
cometric ordering, unless possibly when $d=3$ and the scheme is
isomorphic to the rectangular scheme $R(w,2), w>2$.
\end{corollary}

\noindent We next obtain some (known) results for the parameters of
cometric Q-antipodal schemes. These are used, for example, to show
that the dismantled schemes are also cometric.

\begin{lemma}\label{cometricparameters} A cometric Q-antipodal
scheme has $b_j^*=c_{d-j}^*$ for all $j \neq \lfloor
\frac{d}{2} \rfloor$, $a_j^*=a_{d-j}^*$ for all $j \neq
\frac{d-1}{2}, \frac{d+1}{2}$, and $m_{d-j}=(w-1)m_j$ for
$j<\frac{d}{2}$. Moreover, for $j = \lfloor \frac{d}{2}
\rfloor$, it holds that $b_j^*=(w-1)c^*_{d-j}$.
\end{lemma}

\begin{proof}
From the fact that $E_1 \circ F_j \in \cB$, it follows that  this matrix
is a linear combination of the $F_i$. From the expressions of
$E_1 \circ E_j$ and $E_1 \circ E_{d-j}$ in terms of Krein
parameters and idempotents, we then find that $b_j^*=c_{d-j}^*$
and $a_j^*=a_{d-j}^*$ for all $j \neq \lfloor \frac{d}{2}
\rfloor$. It follows from Lemma \ref{cosetssize2krein} that
$m_{d-j}=(w-1)m_j$ for $j<\frac{d}{2}$.

For odd $d$, and $j=\frac{d-1}{2}$, we have that
$b_j^*=\frac{m_{j+1}}{m_j}c^*_{j+1}=(w-1)c^*_{d-j}$. For even
$d$, and $j=\frac{d}{2}$, we have
$b_j^*=\frac{m_{j+1}}{m_j}c^*_{j+1}=(w-1)\frac{m_{j-1}}{m_j}b^*_{j-1}=(w-1)c^*_{j}=(w-1)c^*_{d-j}$.
\end{proof}

\noindent Before we compute the Krein parameters of the subscheme,
we determine the dual intersection matrix $L^*_d$ and the values of
$\rho_j$. These follow immediately from Proposition
\ref{Qantikrein}.

\begin{lemma}
\label{Lqijk}
The Krein parameters of a cometric Q-antipodal scheme satisfy
the following properties:
\begin{enumerate}
\item[(i)] $q^{j}_{d,d-j}=w-1$ for
$j \leq \frac{d}{2}$;
\item[(ii)]  $q^{j}_{d,d-j}=1$ and $q^{j}_{dj}=w-2$ for
$j > \frac{d}{2}$;
\item[(iii)] $q^i_{dj}=0$ for all other values of $i$
and $j$.
\end{enumerate}
Moreover, $\rho_j=1$ if $j< \frac{d}{2}$, $\rho_{\frac{d}{2}}=w$, and
$\rho_j=w-1$ if $j> \frac{d}{2}$.
\end{lemma}

\noindent For convenient reference, we also collect here a few
equations involving the remaining Krein parameters that were
obtained in Section \ref{uniformschemes} above.

\begin{lemma}
\label{Lqijk2}
The Krein parameters of a cometric Q-antipodal scheme satisfy
the following properties: if $0\le j< \frac{d}{2}$ and $0\le i\le d$, then
\begin{enumerate}
\item[(i)] $q^h_{i,d-j}=(w-1)q^{d-h}_{ij}$ for $h \le \frac{d}{2}$;
\item[(ii)] $q^h_{i,d-j}=q^{d-h}_{ij}+(w-2)q^h_{ij}$ for $h > \frac{d}{2}$; and
\item[(iii)] $q^h_{i,\frac{d}{2}}=q^{d-h}_{i,\frac{d}{2}}$ for all $h$ when
    $d$ is even.
\end{enumerate}
\end{lemma}

\subsection{Subschemes}

\noindent Lemma \ref{subkrein} can now be used to show that the subschemes are cometric (which is analogous to a result
on the folded graph of an antipodal distance-regular graph, see \cite[Prop. 4.2.2.ii]{bcn}).

\begin{proposition}\label{Kreinarray} Let $(X,\cR)$ be a cometric Q-antipodal association scheme with $w$ fibres, and Krein array
$\{b_0^*,b_1^*,\dots,b_{d-1}^*;c_1^*,c_2^*,\dots,c_d^*\}$, where $d
\geq 3$. Then the subschemes induced on the fibres are cometric
with Krein array
$$\{b_0^*,b_1^*,\dots,b_{\frac{d-1}{2}-1}^*;c_1^*,c_2^*,\dots,c_{\frac{d-1}{2}}^*\}$$
for $d$ odd, and Krein array
$$\{b_0^*,b_1^*,\dots,b_{\frac{d}{2}-1}^*;c_1^*,c_2^*,\dots,c_{\frac{d}{2}-1}^*,wc_{\frac{d}{2}}^*\}$$
for $d$ even.
\end{proposition}

\begin{proof} We use Lemma \ref{subkrein} with $i'=i=1$ and $j'=j \leq \frac{d}{2}$ and $h \leq \frac{d}{2}$.

First, let $h < \frac{d}{2}-1$. Then we have that $q^{d-h}_{1j}=0$ because the
scheme is cometric, and hence $\tilde{q}^{h}_{1j}=0$ if $j<h-1$ or $j>h+1$.
Moreover, $\tilde{q}^{h}_{1,h-1}=q^{h}_{1,h-1}=c^*_h$, and
$\tilde{q}^{h}_{1,h+1}=q^{h}_{1,h+1}=b^*_h$. Similarly, it follows for $h \geq
\frac{d}{2}-1$ that $\tilde{q}^{h}_{1j}=0$ if $j<h-1$.

For $h=\frac{d}{2}-1$, we have
$\tilde{q}^{h}_{1,h-1}=q^{h}_{1,h-1}=c^*_h$, and
$\tilde{q}^{h}_{1,h+1}=\frac{1}{w}(q^{h}_{1,h+1}+(w-1)q^{d-h}_{1,h+1})=\frac{1}{w}(b^*_h+(w-1)c^*_{d-h})=b^*_h$.

For $h=\frac{d}{2}$, we obtain that
$\tilde{q}^{h}_{1,h-1}=wq^{h}_{1,h-1}=wc^*_h$, and finally, for
$h=\frac{d-1}{2}$, we obtain that
$\tilde{q}^{h}_{1,h-1}=q^{h}_{1,h-1}=c^*_h$. Thus, it follows
that the scheme is cometric, and the Krein array follows.
\end{proof}

\noindent Note that it follows from the proof that the
Q-polynomial ordering of idempotents is $F_0,F_1,\dots,F_{\lfloor
\frac{d}{2} \rfloor}$. The multiplicities $\tilde{m}_j =
\text{~rank~}F_j$ of a subscheme follow for example as follows:
$\tilde{m}_j=\tilde{q}^0_{jj}=\frac{1}{w}(q^0_{jj}+q^0_{d-j,d-j})=m_j$
for $j \neq \frac{d}{2}$, and
$\tilde{m}_{\frac{d}{2}}=\frac{1}{w}m_{\frac{d}{2}}$.

\subsection{Dismantled schemes}

Proposition \ref{Kreinarray} is a well-known result. In \cite[Thm. 4.7]{mmw} it was shown that a cometric Q-antipodal scheme is
dismantlable, with its dismantled schemes being cometric
Q-antipodal too. The proof of the latter is not complete
however, because incorrect idempotents are suggested there. The
fact that such a dismantled scheme is Q-Higman is clear from
Corollary \ref{Qantipodaluniform}. That it is cometric
Q-antipodal can be shown as follows using Corollary
\ref{dismantledidempotents}.

\begin{theorem}\label{Kreinarraydismantled} Let $(X,\cR)$ be a
cometric Q-antipodal association scheme with $w$ fibres, and
Krein array
$\{b_0^*,b_1^*,\dots,b_{d-1}^*;c_1^*,c_2^*,\dots,c_d^*\}$, where $d
\geq 3$, and let $\ell=\lceil \frac{d}{2} \rceil$. Then the
dismantled scheme induced on a union $Y$ of $w' \geq 2$ fibres
is cometric Q-antipodal with Krein array
$\{\overline{b}_0^*,\overline{b}_1^*,\dots,\overline{b}_{d-1}^*;\overline{c}_1^*,\overline{c}_2^*,\dots,\overline{c}_d^*\},$
where $$\overline{c}_{j}^* = c_{j}^* \text{~for~} j \neq \ell
\text{,~and~} \overline{c}_{\ell}^* =
\frac{w}{w'}c_{\ell}^*~,$$
$$\overline{b}_{j}^* = b_{j}^* \text{~for~} j \neq d-\ell
\text{,~and~}
\overline{b}_{d-\ell}^*=\frac{w}{w'}\frac{w'-1}{w-1}b_{d-\ell}^*~.$$
\end{theorem}

\begin{proof}
The stated result follows from working out the products
$\overline{E}_1 \circ \overline{E}_j$ for all $j$, where we use
the expressions for $\overline{E}_j$ in Corollary
\ref{dismantledidempotents}, and the expressions for the dual
intersection numbers $b_j^*, a_j^*$, and $c_j^*$ in Lemma
\ref{cometricparameters}. For most cases this is rather
straightforward; for readability we will therefore only give
the details of one of the more complicated cases, i.e., that of
$d$ even and $j=\ell+1$. In this case, with
$v'=w'n=\frac{w'}{w}v$ being the number of vertices in $Y$, we
have that
\begin{align*}
v'\overline{E}_1 \circ \overline{E}_{\ell+1} & = vE_1^Y
\circ(E^Y_{\ell+1}+ \frac{w'-w}{w'}E^Y_{\ell-1})\\
 & = b_{\ell}^*E_{\ell}^Y+a_{\ell+1}^*E_{\ell+1}^Y+c_{\ell+2}^*E_{\ell+2}^Y\\
&\quad + \frac{w'-w}{w'}(b_{\ell-2}^*E_{\ell-2}^Y+a_{\ell-1}^*E_{\ell-1}^Y+c_{\ell}^*E_{\ell}^Y)\\
 & = (b_{\ell}^*+\frac{w'-w}{w'}c_{\ell}^*)E_{\ell}^Y + a_{\ell+1}^*(E^Y_{\ell+1}+ \frac{w'-w}{w'}E^Y_{\ell-1})\\
&\quad +c_{\ell+2}^*(E_{\ell+2}^Y+ \frac{w'-w}{w'}E^Y_{\ell-2})\\
 & = \frac{w}{w'}\frac{w'-1}{w-1}b_{\ell}^*\overline{E}_{\ell} + a_{\ell+1}^*\overline{E}_{\ell+1} + c_{\ell+2}^*\overline{E}_{\ell+2}.
\end{align*}
Because $\ell=d-\ell$, it thus follows that
$\overline{b}_{d-\ell}^*=\frac{w}{w'}\frac{w'-1}{w-1}b_{d-\ell}^*$,
$\overline{a}_{\ell+1}^* = a_{\ell+1}^*$, and
$\overline{c}_{\ell+2}^* = c_{\ell+2}^*$. The other parameters
follow similarly, and prove the statement.
\end{proof}

\begin{corollary} Let $(X,\cR)$ be a
cometric Q-antipodal $d$-class association scheme with $w \geq
3$ fibres, with $d$ odd and $\ell=\frac{d+1}{2}$. Then
$a_{\ell}^* \neq 0$. Moreover, if $Y$ is a union of $w'$
fibres, where $w>w' \geq 2$, then $\overline{a}_{\ell-1}^* \neq
0$.
\end{corollary}

\begin{proof} If $w>w' \geq 2$, then $\overline{a}_{\ell}^*=\overline{b}_{0}^*-\overline{c}_{\ell}^*-\overline{b}_{\ell}^*=
b_0^*-\frac{w}{w'}c_{\ell}^*-b_{\ell}^*<a_{\ell}^*$, and
similarly $\overline{a}_{\ell-1}^*>a_{\ell-1}^*$. The result
follows from these inequalities.
\end{proof}

\noindent So, if $d$ is odd, and the scheme is cometric
Q-antipodal Q-bipartite, then $w=2$. Moreover, it cannot be a
dismantled scheme of a cometric Q-antipodal scheme with more
fibres.

\subsection{The natural ordering of relations}\label{natural
ordering}

For a cometric scheme, we define the natural ordering of
relations as the one satisfying $Q_{01}>Q_{11}> \cdots
>Q_{d1}$. Recall that $Q_{ij}A_i=vE_j \circ A_i$. Because $\sum_{i
\in \cI} A_i=n(E_0+E_d)=I_{w}\otimes J_n$ for Q-antipodal
schemes, it follows that in this case $Q_{id}$ equals $w-1$ if
$i \in \cI$, and $-1$ otherwise.

The orthogonal polynomials $q_j, j=0,1,\dots,d+1$ associated to the cometric scheme have the property that
$Q_{ij}=q_j(Q_{i1}), j=0,1,\dots,d$ and $q_{d+1}(Q_{i1})=0$. Because the roots of $q_j$ and $q_{j+1}$ interlace (a
standard and easily proven property of orthogonal polynomials, cf. \cite[Thm. 5.3]{orthog}), it follows that the values
of $Q_{id}$ alternate in sign. Thus for cometric Q-antipodal schemes it follows that $\cI=\{0,2,4,\dots.\}.$

\section{Three-class uniform schemes; linked systems of symmetric
designs}\label{sec:threeclass}

Every two-class imprimitive association scheme is uniform and cometric. It has
one (nontrivial) relation within the fibres and one across the fibres (it is a
wreath product of two trivial schemes), and may thus be seen as a linked system
of complete designs. Likewise, an imprimitive three-class scheme with one
relation across the fibres is uniform (and decomposable), but such a scheme
clearly cannot be cometric.

It is well-known that (homogeneous) linked systems of symmetric designs give
three-class association schemes, and in fact, these are uniform, almost by
definition, and cometric Q-antipodal (for information on such linked systems we
refer to \cite{vandam}, \cite{mmw}, and the references therein). In \cite[Thm.
5.8]{vandam} it was conversely shown (in a different context though) that
imprimitive indecomposable three-class schemes with one extra condition on the
multiplicities must come from such linked systems. We can derive this easily
now from the results in the previous sections.

Indeed, let us consider a three-class imprimitive association scheme that is
indecomposable. Such a scheme must have two relations across the fibres and
have a trivial quotient scheme. Thus we may assume that $\cJ=\{0,3\}$,
$\cJ_1=\{1,2\}$, and $\cI=\{0,2\}$. Moreover we may assume that $m_2 \geq m_1$.
It then follows that the scheme is uniform (Q-Higman) if and only if
$m_2=(w-1)m_1$ (which is the case if and only if $m_1=n-1$). It is clear
(straight from the definition) that such a uniform scheme corresponds to a
linked system of symmetric designs. We thus obtain the same result as in
\cite[Thm. 5.8]{vandam}. The eigenmatrices of a three-class uniform scheme can
be written as

\begin{equation*}\label{P-matrix3}
P = \begin{bmatrix}
1 & (w-1)k_1 & n-1 & (w-1)(n-k_1) \\
1 & P_{11} & -1 & -P_{11}  \\
1 & -\frac{1}{w-1}P_{11} & -1 & \frac{1}{w-1}P_{11} \\
1 &-k_1& n-1 & -(n-k_1)
\end{bmatrix} \end{equation*}
and
\begin{equation*}\label{Q-matrix3}
Q = \begin{bmatrix}
1 & n-1 &  (w-1)(n-1) & w-1 \\
1 & Q_{11} & -Q_{11} & -1 \\
1 & -1 & -(w-1) &  w-1 \\
1 & -\frac{k_1}{n-k_1}Q_{11} & \frac{k_1}{n-k_1}Q_{11} & -1
\end{bmatrix},
\end{equation*}
where $k_1$ is the block size of the symmetric designs in the
corresponding linked system. If we order the relations such that
$P_{11}>0$, then $P_{11}=(w-1)\sqrt{\frac{k_1(n-k_1)}{n-1}}$ and
$Q_{11}=\sqrt{\frac{(n-1)(n-k_1)}{k_1}}$. We remark that Noda
\cite[Prop. 0]{noda} showed that $\frac{k_1(n-k_1)}{n-1}$ is a
square (integer) if $w \geq 3$.

Because the equality $m_2=(w-1)m_1$ is equivalent to $q^3_{11}=0$, it follows
that such a uniform scheme is cometric except possibly when $k_1=1$ (note that
$q^{3}_{12}>0$ because $1 \sim^* 2$; and $q^2_{11}>0$ follows except when
$k_1=1$; we omit the derivation). In case $k_1=1$ however, the scheme is
decomposable: it is a rectangular scheme $R(w,n)$ (the direct product of two
trivial schemes), which is cometric (and metric) if and only if exactly one of
$w$ and $n$ equals $2$. We thus conclude the following.

\begin{proposition}\label{Qpoly3}
Consider an imprimitive three-class association scheme that is
indecomposable, and assume without loss of generality that
$\cJ=\{0,3\}$, $\cI=\{0,2\}$, and $m_2 \geq m_1$. Then it is
uniform if and only if $m_2=(w-1)m_1$. If so, then it is
cometric Q-antipodal and corresponds to a linked system of
symmetric designs.
\end{proposition}

\noindent Davis, Martin, and Polhill \cite{DMP} recently
constructed new linked systems of symmetric designs by using difference sets.
These have the same parameters as the classical ones arising from Kerdock codes
\cite{cs}. A recent construction by Holzmann, Kharaghani, and Orrick \cite[Thm.
2.7]{hadi} of real unbiased Hadamard matrices can be used to construct linked
systems of symmetric designs with new parameters. More precisely, starting from an arbitrary Hadamard matrix of
order $2u$ and an arbitrary set of $w-1$ mutually orthogonal Latin squares of side $2u$, they
construct $w-1$ mutually unbiased regular Hadamard matrices of order $4u^2$. Because these Hadamard matrices are regular,
they correspond to symmetric $2$-$(4u^2,2u^2-u,u^2-u)$ designs, and one obtains a linked system of $w-1$ such designs,
and hence a uniform scheme with $w$ fibres of size $4u^2$.

\section{Four-class cometric Q-antipodal association
schemes}\label{sec:four}

We next consider the four-class schemes, comparing the ``class I''
imprimitive schemes of Higman with the cometric Q-antipodal schemes.

\subsection{A linked system of Van Lint-Schrijver partial
geometries}\label{sec:vls}

Uniform association schemes with three classes and more than
one relation across fibres thus turn out to be cometric. For
four classes this is not the case. There are several examples
with just two fibres that are not cometric, such as those
(non-cometric) schemes generated by bipartite distance-regular
graphs with diameter four. The following example of a system of
linked partial geometries by Cameron and Van Lint \cite{cvl} is
perhaps more interesting because it has three fibres.

\begin{example}
Consider
the ternary repetition code $C$ of length $6$. The vertices of
the association scheme are the $243$ cosets of $C$ in $GF(3)^6$,
and these can be partitioned into three fibres according to the
sum of the coordinates of any vector in the coset. Consider the
graph where two cosets in different fibres are adjacent if one
can be obtained from the other by adding a vector of weight
one. This defines one of the two relations across fibres, and
it generates the entire four-class scheme. The incidence
structure between two fibres is a partial geometry that is
isomorphic to the one constructed by Van Lint and Schrijver
\cite{LS} (with parameters $pg(5+1,5+1,2)$), which has as a
point graph (and line graph) a strongly regular graph with
parameters $(81,30,9,12)$; this gives the two (nontrivial)
relations on the fibres. The scheme is not cometric because
$q^1_{13} \neq 0$.
\end{example}

\subsection{Higman's imprimitive four-class schemes}

Higman \cite{Htriality} studied imprimitive four-class association
schemes, and classified these according to the dimensions of the
subalgebras $\cB$ and $\cC$ associated to a fixed imprimitivity
system (or ``parabolic'') as outlined in Section \ref{Subsec:imprim}
above. Since we showed that $\cB$ has dimension $|\cI|$ and $\cC$
has dimension $|\cJ|$, we may say that a four-class scheme falls
into Higman's ``class I'' (relative to a given imprimitivity system)
if it has $|\cI|=3$ and $|\cJ|=2$. It is known that the cometric
Q-antipodal four-class association schemes fall into this ``class
I". In the next section we shall characterize the cometric schemes in this class.

Let us consider a ``class I" scheme. Although Higman ordered relations and
idempotents differently, we will assume (without loss of generality) that
$\cJ=\{0,4\}$ and $\cI=\{0,2,4\}$. Then, using Lemma \ref{idempotents}, we may
assume that $\cJ_1=\{1,3\}$ and $\cJ_2=\{2\}$.  So the subscheme on each fibre
is a strongly regular graph, on $n$ vertices with valency $k$, say. Let $r$ and
$s$ denote the nontrivial eigenvalues of this graph and let $f$ and $g$ denote
the multiplicities of $r$ and $s$, respectively. The eigenmatrices $\tilde{P}$
and $\tilde{Q}$ for this strongly regular graph are related to the
eigenmatrices of this four-class scheme by Equation \eqref{EPQtilde}. Using
this, we claim (and Higman \cite{Htriality} obtained the same) that the
eigenmatrices for a ``class I'' scheme can be written as

\begin{equation}\label{P-matrix}
P = \begin{bmatrix}
1 & (w-1)k_1 & k & (w-1)(n-k_1) & n-1-k \\
1 & P_{11} & r & -P_{11} & -1-r \\
1 & 0 & s & 0 & -1-s \\
1 & -\frac{m_1}{m_3}P_{11} & r & \frac{m_1}{m_3}P_{11} & -1-r \\
1 &-k_1& k & -(n-k_1)& n-1-k
\end{bmatrix} \end{equation}
and
\begin{equation}\label{Q-matrix}
Q = \begin{bmatrix}
1 & m_1 & wg & m_3 & w-1 \\
1 & Q_{11} & 0 & -Q_{11} & -1 \\
1 & \frac{m_1}{k}r & \frac{wg}{k}s & \frac{m_3}{k}r & w-1 \\
1 & -\frac{k_1}{n-k_1}Q_{11} & 0 & \frac{k_1}{n-k_1}Q_{11} & -1 \\
1 & -\frac{m_1}{n-1-k}(1+r) & -\frac{wg}{n-1-k}(1+s) &
-\frac{m_3}{n-1-k}(1+r) & w-1
\end{bmatrix},
\end{equation}
where $k_1:=1+p^1_{12}+p^1_{14}$ and the remaining
unknowns are related by
$$ m_1+m_3 = wf, \qquad  P_{11}m_1 = Q_{11}v_1, \qquad v_1=(w-1)k_1.$$
Indeed, for a given vertex $x$ and a fibre
$U$ not containing $x$,  $k_1$ equals the number of
1-neighbors of $x$ in $U$. So the incidence structure
between any two fibres induced by relation $R_1$ is a square 1-design with
block size $k_1$.  Thus the total number of 1-neighbors of $x$ equals
$v_1=(w-1)k_1$.  We also have $Q_{12}=Q_{32}=0$ because
$\cJ_2=\{2\}$ forces $E_2 \in \cB$.
The remaining simplifications  in \eqref{P-matrix}  and \eqref{Q-matrix}
can easily be checked using the orthogonality relations $Q_{ij}=P_{ji}\frac{m_j}{v_i}$
and (column zero of) $PQ=QP=vI$.

It will benefit us to make the expressions \eqref{P-matrix} and
\eqref{Q-matrix} as unambiguous as possible.
Let us agree to order the idempotents $E_1$ and $E_3$ by $m_1 \leq
m_3$. Unless otherwise noted, we will order the relations $R_2$ and
$R_4$ by assuming that $r \geq 0$, and the relations
$R_1$ and $R_3$ by assuming $P_{11} \geq 0$. We now verify that,
if such a scheme is cometric, then $E_0,E_1,E_2,E_3,E_4$ must be the Q-polynomial
ordering, except possibly when $w=2$.

Since columns two and four of $Q$ have repeated entries, neither $E_2$ nor
$E_4$ can be a Q-polynomial generator. In fact, $E_4$ must take the last
position in any Q-polynomial ordering by the same argument as that in the proof
of Theorem \ref{cometricuniformQantipodal}. Finally, $E_2$ cannot take position
three because $q_{13}^4
> 0$ follows from $1\sim^* 3$. The last
two possibilities for our Q-polynomial ordering are $E_0,E_1,E_2,E_3,E_4$ and
$E_0,E_3,E_2,E_1,E_4$. But $q_{43}^3=0$ then gives $m_1=(w-1)m_3$ in the second
case (by Lemma \ref{cosetssize2krein}) and, with our conventions above, this
can only happen if $w=2$. In fact, when $w=2$, we find that either one of these
orderings -- or both of them -- can be Q-polynomial orderings. But in the case
where $E_3$ is the Q-polynomial generator, the natural ordering of relations
described in Section \ref{natural ordering} is instead $R_0,R_3,R_2,R_1,R_4$.

From the 13-entry of the equation $PQ=vI$ and the 11-entry from
the similar equation for the subscheme, we find that $$P_{11}=
\sqrt{\frac{m_3(w-1)k_1(n-k_1)}{m_1f}}.$$ By using the
expression \cite[Thm.~II.3.6(i)]{banito}
\begin{equation}
\label{Eqijk}
q^h_{ij}=\frac{m_im_j}{v}\sum_l
\frac{P_{il}P_{jl}P_{hl}}{v_l^2},
\end{equation}
and the similar expression
$$\tilde{q}_{11}^1 = \frac{f^2}{n} \left( 1 + \frac{ r^3 }{k^2}
- \frac{ (1+r)^3}{(n-1-k)^2} \right)$$
for the subscheme we then derive that
$$q^1_{13}=\frac{m_1m_3}{wf^2}\left(\tilde{q}^1_{11}-\sqrt{\frac{m_3}{m_1(w-1)}}\frac{(n-2k_1)\sqrt{f}}{\sqrt{k_1(n-k_1)}}\right),$$
which, of course, must vanish when the
scheme is cometric with respect to the ordering $E_0,E_1,E_2,E_3$, $E_4$.

\subsection{Linked systems of strongly regular
designs}\label{sec:linkedsrd}

Let us proceed with the expressions of the previous section. From
Lemma \ref{cosetssize2krein}, we know that
$q^1_{14}=0$ if and only if $m_3=m_1(w-1)$. By Definition \ref{def:Qanti}
and Theorem \ref{TuniformiffQHigman}, this happens if and only if
the scheme is uniform. In this case, the incidence structure between
two fibres is a so-called strongly regular design as defined by
Higman \cite{Hsrd}, and the scheme corresponds to a linked system of
strongly regular designs. Cameron and Van Lint \cite{cvl}
constructed such an example, as we saw, and also the example in
Section \ref{HOSI} is a linked system of strongly regular designs.

\begin{proposition}\label{Qpoly4}
An imprimitive four-class association scheme of Higman's
``class I'' is cometric (and therefore Q-antipodal) if and only if $r \neq k$,
$m_3=(w-1)m_1$, and
\begin{equation}\label{q_condition2}
\tilde{q}_{11}^1 = \frac{(n-2k_1)\sqrt{f}}{\sqrt{k_1(n-k_1)}},
\end{equation}
possibly after reordering the idempotents $E_1$ and $E_3$ and the
relations $R_1$ and $R_3$ in the case $w=2$.
\end{proposition}

\begin{proof}
We address the case $w\ge3$. The same ideas work in the case where
$w=2$, but an extra case argument is involved.

First recall that a cometric Q-antipodal scheme is uniform and we have
just shown that uniformity, the vanishing of $q_{14}^1$, and the equation
$m_3=(w-1)m_1$ are all equivalent.

We know from above that the scheme is cometric if and only if $E_0,E_1,
E_2,E_3,E_4$ is a Q-polynomial ordering. So we need
$$ q^1_{14}=0 , \quad   q^1_{13}=0, \quad  q^2_{11} > 0, \quad
q^3_{12}>0, \quad q^4_{13}>0  .  $$
Observing that
$q^2_{11}=\frac{m_1^2}{wf^2}\tilde{q}^2_{11}$ and
$q^2_{13}=\frac{m_1m_3}{wf^2}\tilde{q}^2_{11}$, and that
$\tilde{q}^2_{11}=0$ if and only if the strongly regular graphs on
the fibres are imprimitive with $r=k$,  one easily works out the
remaining implications in both directions.
\end{proof}

\noindent Thus, for a cometric Q-antipodal four-class
association scheme, all parameters can be expressed in terms
of the number of fibres, $w$, and the parameters of the strongly
regular graph.

\begin{corollary}
\label{CQantipparams} If $(X,\cR)$ is a cometric Q-antipodal four-class
association scheme with $w$ fibres, then there exists a strongly regular graph
with $n=v/w$ vertices, eigenvalues $k$, $r$, and $s$ having multiplicities $1$,
$f$, and $g$ respectively, such that the eigenmatrices for $(X,\cR)$ are given
by Equations \eqref{P-matrix} and \eqref{Q-matrix} where $m_1=f$, $m_3=(w-1)f$,
$$P_{11}= (w-1)\sqrt{k_1(n-k_1)/f}, \qquad  Q_{11}= \sqrt{ f(n-k_1)/k_1 },$$
and
\begin{equation}\label{eqS}
k_1=\frac{n}{2}\left(1-\frac{\tilde{q}_{11}^1}{\sqrt{4f+(\tilde{q}_{11}^1)^2}}\right).
\end{equation}
\end{corollary}

\begin{proof}
The expression for $k_1$ follows from \eqref{q_condition2}.
\end{proof}

\noindent Moreover, because $\tilde{q}_{11}^1$ is always a rational
number (even in the case when some entry of $\tilde{P}$ is irrational),
we obtain that $\sqrt{k_1(n-k_1)/f}\in\Q$ for a
cometric Q-antipodal scheme with four classes, unless perhaps
when $n=2k_1$ (equivalently, $\tilde{q}_{11}^1=0$).
On the other hand $(w-1)\sqrt{k_1(n-k_1)/f} = P_{11}$ is an
algebraic integer. Therefore $P_{11}$ is a rational integer if $n \neq 2k_1$.
Because a Q-antipodal cometric scheme is
dismantlable, we can take $w=2$ and consider $P_{11}$ for the
dismantled scheme; now we see that $k_1(n-k_1)/f$ is a perfect
square provided $n \neq 2k_1$.

It also follows from \eqref{eqS} that if $\tilde{q}_{11}^1 \neq
0$, then $4f+(\tilde{q}_{11}^1)^2$ is a square of a rational
number. This immediately implies the following result, which
we will use in Section \ref{sec:srd}.

\begin{proposition}\label{conference}
For a cometric Q-antipodal four-class association scheme, the strongly regular
graph on a fibre cannot be a conference graph.
\end{proposition}
\begin{proof}
Assume the contrary. Then $n=2k+1$, $f = k > 0$,
$\tilde{q}_{11}^1=(k-2)/2$, $k$ is even, and
$$
4f+(\tilde{q}_{11}^1)^2 = 4k + \left(\frac{k-2}{2}\right)^2.
$$
Because $n$ is odd, $\tilde{q}_{11}^1 \neq 0$. Therefore
$k^2+12k+4$ is the square of an integer. But $k=-12,0$ are the only
even integers for which the expression $k^2+12k+4$ is a perfect
square.
\end{proof}

\noindent The rationality condition that follows from
\eqref{eqS} turns out to be quite a strong one. It is possible
to show, for example, that also the lattice graphs cannot occur
as our strongly regular graph on the fibres, and probably many more graphs can be
excluded in this way. We will employ this condition as well in the next section.

\subsection{\ \ \ Four-class cometric Q-antipodal Q-bipartite association
schemes; linked systems of Hadamard symmetric nets}

Recently, four-class cometric Q-antipodal Q-bipartite
association schemes were shown to be equivalent to so-called
real mutually unbiased bases, and a connection to Hadamard
matrices was found in \cite{mub}. We also refer to \cite{abs}
for connections between real mutually unbiased bases and
association schemes. Here we shall derive the connection to
Hadamard matrices, and see cometric Q-antipodal Q-bipartite
four-class schemes as linked systems of Hadamard symmetric
nets.

So, let us consider a cometric Q-antipodal Q-bipartite four-class association scheme, and its eigenmatrix $Q$ in
\eqref{Q-matrix} with $m_3=(w-1)m_1=(w-1)f$ and $Q_{11}= \sqrt{ f(n-k_1)/k_1 }$ from Corollary \ref{CQantipparams}.
Since the scheme is cometric Q-bipartite, the column of $Q$ corresponding to a Q-polynomial generator has its $d+1$
distinct values symmetric about zero when ordered naturally \cite[Cor. 4.2]{mmw}. In our case, this is either column
one or column three, and in both cases it follows that $r=0$, $n=k+2$, and $n=2k_1$. This implies that $s=-2$,
$f=\frac{n}{2}$, and the strongly regular graphs on the fibres are cocktail party graphs (complements of matchings).
Now restrict to any dismantled scheme on $w'=2$ fibres; straightforward calculations show that this must correspond to
a so-called Hadamard graph, an antipodal bipartite distance-regular graph of diameter four, cf. \cite[p.\ 19,
425]{bcn}. Such graphs correspond to Hadamard matrices; more precisely, the incidence structure between a pair of
fibres is a Hadamard symmetric net (that is, a symmetric $(m,\mu)$-net with $m=2$). We thus obtain that cometric
Q-antipodal Q-bipartite four-class association schemes are linked systems of Hadamard symmetric nets. Interesting
examples of these are given by the extended Q-bipartite doubles of the three-class uniform schemes corresponding to the
known linked systems of symmetric designs of Section \ref{sec:threeclass}. We expect that the schemes that arise in
this way from the construction by Holzmann, Kharaghani, and Orrick \cite[Thm. 2.7]{hadi} of real unbiased Hadamard
matrices are the same as those coming from the mutually unbiased bases constructed by Wojcan and Beth \cite{wb}, but we
have not checked the details. For more on the correspondence to real mutually unbiased bases, and bounds on $w$, we
refer to \cite{mub}.

On the other hand, we can characterize the cometric Q-antipodal
Q-bipartite four-class association schemes as follows.

\begin{proposition}\label{Qbipartite}
Consider a cometric Q-antipodal four-class association scheme, such that the
strongly regular graph on a fibre is imprimitive. Then it is Q-bipartite.
\end{proposition}
\begin{proof}
By Proposition \ref{Qpoly4}, we have $r\neq k$. So $r=0$, and the
strongly regular graph on a fibre must be a complete multipartite
graph, say a $t$-partite graph with parts of size $\frac{n}{t}$
each. For such a graph $s=-\frac{n}{t}$, $f=n-t$, and
$\tilde{q}_{11}^1 = n-2t$. So, if $\tilde{q}_{11}^1 \neq 0$
(which is equivalent to $s \neq -2$), then $t \leq
\frac{n}{3}$, and $4f+(\tilde{q}_{11}^1)^2$ is square (as
before by \eqref{eqS}). However, for $t \leq \frac{n}{3}$, we
have $(n-2t+2)^2+4t-4=4f+(\tilde{q}_{11}^1)^2<(n-2t+4)^2$, so
$4f+(\tilde{q}_{11}^1)^2$ cannot be square. Thus,
$\tilde{q}_{11}^1 = 0$ and $t =\frac{n}{2}$, so the strongly
regular graph on a fibre is a cocktail party graph, and
therefore $n=k+2$ and $n=2k_1$.
From the expression for the eigenmatrix $P$ in \eqref{P-matrix}
and Equation \eqref{Eqijk}, one
can now derive that the Krein parameters $a_i^*=q^{i}_{1i}$ are
zero for all $i$. Thus the scheme is Q-bipartite. Note that,
in this case, not only is column one, but also is column three of $Q$
symmetric about zero.
\end{proof}

\noindent The same result may be derived by using the fact that
there are two different imprimitivity systems and Suzuki's
results on imprimitive cometric schemes \cite{suzimprim} and
cometric schemes with multiple Q-polynomial orderings
\cite{suztwoq}. It would be interesting to work this out more
generally, that is, for any cometric scheme with multiple
imprimitivity systems, but we leave this to the interested reader.

\subsection{Strongly regular graphs with a strongly regular
decomposition}\label{sec:srd}

One of the interesting features of the example in Section
\ref{HOSI} is that there is a decomposition of the Higman-Sims
graph into two Hoffman-Singleton graphs; thus a strongly
regular graph decomposes into two strongly regular graphs. Such
strongly regular graphs with a strongly regular decomposition
were studied by Haemers and Higman \cite{HH} and Noda
\cite{noda2}, and they occur in more examples of four-class
cometric Q-antipodal association schemes, as we shall see.

Let $\Gamma_0 = (X,E)$ be a primitive strongly regular graph
with adjacency matrix $M$, parameter set
$(v,k_0,\lambda_0,\mu_0)$, and distinct eigenvalues
$k_0>r_0>s_0$. A strongly regular decomposition of $\Gamma_0$
is a partition of $X$ into two sets $U_1$ and $U_2$ such that
the induced subgraphs $\Gamma_i:=\Gamma_{U_i}, i=1,2$ are
strongly regular.

For our purpose, the sets $U_1$ and $U_2$ will play the role of
the $w=2$ fibres of an imprimitive (bipartite) association
scheme, and the disjoint union of the graphs $\Gamma_1$ and
$\Gamma_2$ is one of the two relations in $\cI$. Thus we will
only consider the case that the sets $U_1$ and $U_2$ are of
equal size, and the parameter sets of $\Gamma_1$ and $\Gamma_2$
are the same, say $(n,k,\lambda,\mu)$. The eigenvalues of both
graphs will be denoted by $k \geq r > s$.

To make the connection between a strongly regular graph with a
strongly regular decomposition and our four-class association
schemes more precise, write $M=\mtrx{M_1}{C}{C^\top}{M_2}$, where
the blocks correspond to our partition of $X$. We then define
relations by the following adjacency matrices:
\begin{equation}
\label{Esrdecomp}
\begin{split}
A_0:=\mtrx{I}{0}{0}{I}, \quad A_1:=&\mtrx{0}{C}{C^\top}{0},
\quad A_2:=\mtrx{M_1}{0}{0}{M_2}, \\
A_3:=\mtrx{0}{J-C}{J-C^\top}{0}, \quad & A_4:=\mtrx{J-M_1-I}{0}{0}{J-M_2-I}.
\end{split}
\end{equation}
We shall determine when these relations form an association
scheme, and if they do, we shall see that the scheme is cometric
Q-antipodal. But first we make some more observations.

By taking the complements of $\Gamma_i$, $i=0,1,2$, we obtain
another strongly regular graph with a strongly regular
decomposition; we call this the complementary decomposition.
Note that this complementary decomposition determines the same
relations, i.e., the same $A_i$, $i=0,\dots,4$, but ordered differently. In case
that these relations form an association scheme, it is not
clear a priori which ordering corresponds to the one in the
eigenmatrix $P$ in \eqref{P-matrix}. The straightforward choice
that we make is that we consider that decomposition for which
the eigenvalues $k,r,s$ of $\Gamma_i,i=1,2$ correspond to the
$k,r,s$ in the eigenmatrix $P$ (however, in the case of hemisystems
in the next section we make an exception).

A strongly regular decomposition is called exceptional if
$r_0\neq r$ and $s_0\neq s$. It was shown by Haemers and Higman
\cite[Thm. 2.7]{HH} (and it also follows from \cite[Thm.
1]{noda2}) that in this exceptional case the graphs $\Gamma_1$
and $\Gamma_2$ are conference graphs. Thus,
Proposition~\ref{conference} implies that such an exceptional
decomposition does not correspond to a cometric scheme. An
example of an exceptional decomposition is that of the Petersen graph
into two pentagons.

Note that when the relations defined by \eqref{Esrdecomp}
do form an association scheme, then it has a fusion scheme
$\{A_0,A_1+A_2,A_3+A_4\}$.
In that case it follows from the expression \eqref{P-matrix}
for the eigenmatrix $P$
that the strongly regular graph $\Gamma_0$
with adjacency matrix $M=A_1+A_2$ has an eigenvalue $0+s$,
hence $s_0=s$, and the decomposition is not exceptional. Note
that in \eqref{P-matrix} the roles of
$A_1$ and $A_3$ may be swapped, but this has no
influence on the observation.
Thus, in the case of an exceptional decomposition, \eqref{Esrdecomp}
does not yield an association scheme.

We shall now show that if the relations defined by \eqref{Esrdecomp}
form a scheme, then this scheme is cometric Q-antipodal.
This follows from the following
proposition, where we consider Higman's ``class I" schemes with
two fibres, i.e., $w=2$; since the scheme is bipartite, it is
Q-Higman and Lemma \ref{cosetssize2krein} applies,
giving us $m_4=1$, $m_1=m_3=f$ and
$q^j_{4,4-j}=1$. Thus, $q^1_{41}=q^3_{43}=0$, and the conditions
from Proposition \ref{Qpoly4} for the scheme to be cometric reduce
to $r \neq k$ and $q^3_{11}=0$ or $q^1_{33}=0$.

\begin{proposition}\label{srd}
Consider an imprimitive four-class association scheme in
Higman's ``class I" with two fibres. Suppose that the scheme
has a primitive two-class fusion scheme, and that $r \neq k$.
Then the scheme is cometric Q-antipodal.
\end{proposition}
\begin{proof}
From the form of the eigenmatrix $P$ in \eqref{P-matrix} it follows that the
only way to obtain a primitive two-class fusion (i.e., one where both
nontrivial relations correspond to connected
strongly regular graphs) is to fuse relation $R_1$ with either
$R_2$ or $R_4$, and to fuse the remaining two nontrivial
relations. But then there exists a corresponding partition $\{T_1,T_2\}$ of
$\{1,2,3,4\}$ such that $E_0$, $E_{T_1}:=\sum_{j\in T_1}E_j$
and $E_{T_2}:=\sum_{j\in T_2}E_j$ are the primitive
idempotents of the fusion scheme. Depending on the fusion of
relations, one of the fused relations has eigenvalue
$\pm P_{11}+r$ corresponding to idempotent $E_1$, and eigenvalue
$\mp P_{11}+r$ corresponding to idempotent $E_3$, these two eigenvalues
differing by $2P_{11}$ in either case.  In any case it
follows that $1$ and $3$ are not in the same set $T_i$.

Now assume first that one of $T_1,T_2$ is a one-element set,
say $T_1=\{i\}$. From the above it follows that $i \neq 2,4$.
If $i=1$, then $E_1\circ E_1$ is a linear combination of
$E_0,E_1,E_2+E_3+E_4$. But $q_{11}^4=0$. Therefore $E_1\circ
E_1\in\langle E_0,E_1\rangle$ implying that the fusion is
imprimitive, which is a contradiction. The case $i=3$ can be
settled analogously.

Thus $|T_1|=|T_2|=2$. Without loss of generality $T_1=\{i,4\}$
for $i=1$, or $i=3$ (the case $i=2$ is eliminated by the above
considerations). Assume without loss of generality that $i=1$;
then
$$
v(E_1 + E_4)\circ (E_1+E_4) = (m_1+1)E_0 + x (E_1+E_4)+ y (E_2+E_3)
$$
for some non-negative reals $x,y$. Because $q_{11}^4=0$,
$q_{14}^4=0$, and $q_{44}^4=0$ (by Proposition \ref{Qantikrein} and
using $w=2$), $E_4$ does not appear in the left-hand side. Therefore $x=0$,
implying $E_1\circ E_1\in \langle E_0,E_2,E_3\rangle$. Together with
$E_3\circ E_3 = E_1\circ E_1$ (which follows from the equations
$vE_4\circ E_4 = E_0$ and $vE_3\circ E_4 = E_1$) we obtain $E_3\circ
E_3\in \langle E_0,E_2,E_3\rangle$. So $q^1_{33}=0$, which yields
the claim.
\end{proof}

\noindent What remains is to show that a decomposition that is
not exceptional gives an association scheme. This gives the
following result.

\begin{proposition}\label{srd_Qpoly} Consider a primitive strongly regular graph with a strongly
regular decomposition into parts with the same parameters. Then
the above-mentioned relations form an association scheme if and
only if the decomposition is not exceptional. If so, then for
$r \neq k$, the scheme is cometric Q-antipodal.
\end{proposition}

\begin{proof} We showed before that an exceptional
decomposition does not correspond to an association scheme. So
suppose that the decomposition is not exceptional. From the
parameters of the strongly regular graphs it follows that
\begin{align*}
&M^2 = (r_0+s_0)M -r_0s_0 I +(k_0+r_0s_0)J, \qquad MJ = k_0J,\\
&M_i^2 = (r+s)M_i -rs I +(k+rs)J, \qquad \text{~and~}  \qquad M_iJ = kJ, \quad i=1,2.
\end{align*}
By working out the first equation, it follows that
\begin{gather}\label{eq_Lemma21}
\begin{aligned}
&CJ = C^\top J = (k_0 -k) J,  \\
&M_1 C + CM_2 = (r_0+s_0) C +(k_0+r_0s_0) J,\\
&CC^\top = (r_0+s_0-r-s)M_1 - (r_0s_0-rs) I + (k_0 + r_0s_0-k-rs)J,\\
&C^\top C = (r_0+s_0-r-s)M_2 - (r_0s_0-rs) I + (k_0 + r_0s_0-k-rs)J,
\end{aligned}
\end{gather}
and this implies that $(r_0+s_0-r-s)(M_1C-CM_2)=0$. If
$r_0+s_0=r+s$, then it follows from a result of Noda \cite[Thm.
1]{noda2} that the decomposition is exceptional, hence we must
have that $M_1C=CM_2$. From \eqref{eq_Lemma21} it then follows
that $M_1 C = C M_2 = \frac{r_0 +s_0}{2}
C+\frac{k_0+r_0s_0}{2}J$. Now a routine check shows that the
matrices $A_i, i=0,\dots,4$ form an association scheme, and by
Proposition~\ref{srd} this scheme is cometric Q-antipodal.
\end{proof}

\noindent For the non-exceptional case, Noda \cite{noda2} found
that all parameters of the decomposition can be expressed in
terms of $r_0$ and $s_0$. In our case, we have that $s=s_0$,
which is the complementary case to the one considered in \cite[Thm.
1]{noda2}. From this result, it follows for example that
$r=\frac{r_0+s_0}{2}$. Note that this also follows by
considering the eigenvalues of the fusion scheme using
\eqref{P-matrix}: indeed, we have $s_0=0+s=-P_{11}+r$,
and $r_0=P_{11}+r$.

Haemers and Higman \cite{HH} give a list of parameter sets of
non-exceptional decompositions on at most 300 vertices. The
smallest example is the Clebsch graph that decomposes into two
perfect matchings on 8 vertices. The association scheme
corresponding to this decomposition (consider the complementary
one for the parameters) is the four-class binary Hamming scheme
$H(4,2)$ (which is (co-)metric, (Q-)bipartite, (Q-)antipodal).
Note that this is a dismantled scheme of the cometric
Q-bipartite Q-antipodal scheme (with $w=3$) related to the
so-called $24$-cell. The next example is the Higman-Sims graph
decomposing into two Hoffman-Singleton graphs, and there are
two more examples: on 112 vertices and 162 vertices. The one on
112 vertices is a decomposition of a generalized quadrangle
into two Gewirtz graphs, and it is part of an infinite family of
decompositions coming from hemisystems.

\subsubsection{Hemisystems of generalized quadrangles}\label{Sec:hemi}

Segre \cite{segre} introduced the concept of hemisystems on the
Hermitian surface $H$ in $PG(3,q^2)$ as a set of lines of $H$ such
that every point in $H$ lies on exactly $(q+1)/2$ such lines. This
point-line geometry, denoted $H(3,q^2)$, gives an important
classical family of generalized quadrangles, called the Hermitian
generalized quadrangles. It is now well-known \cite{cdg} that the
incidence relation on lines in this hemisystem yields a strongly
regular subgraph of the line graph of the geometry. Thus we obtain a
strongly regular  decomposition of the (strongly regular) line graph
of this generalized quadrangle.  In fact, this holds for any
hemisystem in a generalized quadrangle $GQ(t^2,t)$.

Let $(\cP,\cL)$ be the point-line incidence structure of a
generalized quadrangle $GQ(t^2,t)$ with $t$ odd. Let $\Gamma_0$ be
the line graph: its vertex set is $X=\cL$ with two vertices
adjacent if the lines have a point in common. This is a strongly
regular graph with parameters $( (t^3+1)(t+1), t(t^2+1), \ t-1,
t^2+1)$ and with eigenvalues $k_0=t(t^2+1)$, $r_0= t-1$, and
$s_0=-1-t^2$.
 A hemisystem in $(\cP,\cL)$ is a subset
$U_1 \subseteq \cL$ with the property that every point in $\cP$ lies on exactly
$(t+1)/2$ lines in $U_1$ and $(t+1)/2$ lines in $U_2=X-U_1$.
Cameron, Delsarte, and Goethals \cite{cdg} showed that
any hemisystem in a generalized quadrangle of
order $(t^2,t)$ corresponds to a strongly regular
decomposition of the line graph of the corresponding generalized quadrangle.
Because the complementary set $U_2$ of lines of a hemisystem is also a hemisystem,
this decomposition $X=U_1 \cup U_2$ has equally sized parts. Moreover, the parameters
of the parts are the same: each $U_i$ induces a subgraph $\Gamma_i$ which
is strongly regular with parameters
$$(n,k,\lambda,\mu)= \left(  \frac{1}{2}(t^3+1)(t+1), \frac{1}{2}(t^2+1)(t-1), \
\frac{1}{2}(t-3), \frac{1}{2}(t-1)^2 \right)
$$
and eigenvalues $k=\frac{1}{2}(t^2+1)(t-1)$, $r=t-1$, and
$s=-\frac{1}{2}(t^2-t+2)$. The decomposition is clearly not
exceptional (note though that here we have the complementary
setting as in the previous section because $r=r_0$), so by
Proposition \ref{srd_Qpoly}, we have

\begin{corollary}
\label{Chemisystems} Let $(\cP,\cL)$ be a generalized quadrangle $GQ(t^2,t)$ with $t$ odd and let $\cC$ denote the set
of all ordered pairs of distinct intersecting lines from $\cL$. Suppose $\cL = U_1 \cup U_2$ is a partition of the
lines into hemisystems. Then the relations $R_0 = \{ (\ell,\ell)| \ell\in \cL\}$, $R_1 = \cC \cap ( U_1 \times U_2 \cup
U_2 \times U_1)$, $R_2 = \cC \cap ( U_1 \times U_1 \cup U_2 \times U_2)$, $R_3 = ( U_1 \times U_2 \cup U_2 \times U_1)
- R_1$, $R_4 = ( U_1 \times U_1 \cup U_2 \times U_2) - R_0 -R_2$ give a cometric Q-antipodal association scheme on
$X=\cL$. This scheme has Krein array $ \left\{ (t^2+1)(t-1),  (t^2-t+1)^2/t, (t^2-t+1)(t-1)/t,  1 \right.$; $1,
(t^2-t+1)(t-1)/t, (t^2-t+1)^2/t,$ $\left. (t^2+1)(t-1) \right\}.$
\end{corollary}

\noindent Segre \cite{segre} constructed a hemisystem in $H(3,q^2)$ (a
$GQ(q^2,q)$) for $q=3$; it corresponds to the above-mentioned example on 112
vertices with a decomposition into Gewirtz graphs. A breakthrough was made by
Cossidente and Penttila \cite{Penttila}, who constructed hemisystems in
$H(3,q^2)$ for all odd prime powers $q$. Bamberg, De Clerck, and Durante
\cite{Bamberg} constructed a hemisystem for a nonclassical generalized
quadrangle of order $(25,5)$ (which has the same parameters as $H(3,25)$), and
Bamberg, Giudici, and Royle \cite{flock} showed that every flock generalized
quadrangle has a hemisystem. The latter authors \cite{Bamcomp}
recently also classified by computer the hemisystems for the two known
generalized quadrangles of order $(25,5)$, and obtained several examples for
other small generalized quadrangles.

\subsection{Classification, parameter sets, and examples}

We saw in Section \ref{sec:linkedsrd} that the parameters of a
four-class cometric Q-antipodal scheme are completely
determined by those of the strongly regular graph on the
fibres, together with the number of fibres $w$. We used this to
generate ``feasible" parameter sets for four-class cometric
Q-antipodal schemes that are not Q-bipartite, and that have $n
\leq 2000$ and $w \leq 6$. These parameter sets are listed in
the appendix. Standard conditions such as integrality of
parameters $p_{ij}^h$ and nonnegativity of the Krein parameters
$q_{ij}^h$ were checked. Once a parameter set failed, we did
not search for the corresponding parameter set with larger $w$
(because dismantlability would exclude such a parameter set).
We also checked one of the so-called absolute bounds on
multiplicities, i.e., the one in Proposition
\ref{absolutebound} in the next section.

\subsubsection{Absolute bound on the number of fibres}
\label{Subsec:absbound}

By the absolute bound we obtain the following bound for $w$.

\begin{proposition}\label{absolutebound}
For a four-class cometric Q-antipodal scheme with $a_1^* \neq 0$, we have $w \leq (f+1)(f-2)/2g$.
\end{proposition}

\begin{proof}
By the absolute bound (cf. \cite[Thm. 2.3.3]{bcn}) and because $a_1^* \neq 0$, we have
$$f(f+1)/2 \geq \rank(E_1\circ E_1) =\rank(E_0) +\rank(E_1) + \rank(E_2) =
1+f+wg,$$ and the result follows.
\end{proof}

\noindent  For example, for the parameter sets with $n=81$ in
the appendix, we obtain that $w \leq 3$ from $f=20$ and $g=60$.
In general, the bound does not appear to be very good though.

\subsubsection{The small examples}
\label{smallcases}

The first family of parameter sets in the appendix ($n=50$)
corresponds to the examples (with $w=2$ and $w=3$) related to
the Hoffman-Singleton graph in Section \ref{HOSI}. The case
$w=2$ corresponds to a distance-regular graph that is uniquely
determined by the parameters, cf. \cite[p.\ 393]{bcn}. Now
consider more generally an association scheme with $w$ fibres
$V_i, i=1,\dots,w$ (in this family of parameter sets). Because
also the Hoffman-Singleton graph is determined by its
parameters, relation $R_4$ is such a graph on each fibre. Let
us call two vertices in distinct fibres incident if they are
related by relation $R_1$. Because $p^1_{14}=0$ for all $w$, it
follows that if we take a vertex $x \in V_i, i>1$, then the 15
vertices in $V_1$ incident to $x$ will form a coclique in the
Hoffman-Singleton graph on $V_1$. Because distinct $x$ are
incident to distinct cocliques, and there are exactly 100
distinct cocliques of size 15 in the Hoffman-Singleton graph,
it follows that $w \leq 3$. Moreover, because the scheme with
$w=2$ is uniquely determined by its parameters, and is a
dismantled scheme of a scheme with $w=3$, this implies that the
latter scheme is also uniquely determined by its parameters.

For the second family of parameter sets in the appendix
($n=56$) a construction is known for $w=3$. Higman \cite[Ex.
3]{Htriality} for example mentions it can be constructed on the
set of ovals in the projective plane of order $4$. The fibres
are the three orbits of ovals under the action of the group
$L_3(4)$. The case $w=2$ corresponds to a hemisystem of the
generalized quadrangle of order $(9,3)$, or equivalently, to a
strongly regular decomposition of the point graph of $GQ(3,9)$
into two Gewirtz graphs. It is known that such a decomposition,
and hence the corresponding scheme, is unique (the uniqueness
of the hemisystem in the generalized quadrangle is proven by
Hirschfeld \cite[Thm. 19.3.18]{hirschfeld}, and the uniqueness
of the point graph as a strongly regular graph was proven by
Cameron, Goethals, and Seidel \cite{cgssrg}). As in the first
family of parameter sets, we can show here that $w \leq
3$, and that the scheme with $w=3$ is unique. In this case, the
intersection number $p^1_{14}$ equals one (for all $w$), which
implies that the set of 20 neighbors in $V_1$ of any vertex $x
\notin V_1$ must be an induced matching $10K_2$ in the Gewirtz
graph induced on $V_1$. Brouwer and Haemers \cite[p.\ 405]{BH}
mention that there are exactly 112 such induced subgraphs in
the Gewirtz graph, which implies that $w \leq 3$ as well as the
uniqueness of the scheme with $w=3$.

The case $n=64$ has $w \leq 2$. Dismantlability implies that
the schemes with $n=64$ and $w>2$ do not exist (a scheme with
$w=3$ does not occur because for example the intersection
number $p_{11}^1=4.5$ is not integer). The case $w=2$
corresponds to the distance-regular folded 8-cube, which is
uniquely determined by its parameters.

For the family of parameter sets with $n=81$, the absolute bound
implies that $w \leq 3$. Goethals and Seidel \cite[p.\ 156]{GS} give a
decomposition of the strongly regular graph on 243 vertices from the
ternary Golay code (also known as Delsarte graph) into three
strongly regular graphs on 81 vertices. This gives a scheme with
$w=3$ and $n=81$. According to Brouwer \cite{AEBweb}, the
decomposition of the unique strongly 56-regular graph on 162
vertices into two strongly regular graphs on 81 vertices is unique,
hence the association scheme with $w=2$ is unique as well. We also
expect the scheme with $w=3$ to be unique.

Besides the above examples, and the examples related to
triality or hemisystems, there occurs one more family of
examples in the appendix. These are related to the Leech
lattice, cf. \cite[Ex. 4]{Htriality}, and have $n=1408$ and $w
\leq 3$.

Curiously, the Krein array
$\{176, 135, 24, 1; 1, 24, 135, 176 \}$
is formally dual to the intersection array
of a known graph,
a cometric antipodal distance-regular double cover
on 1344 vertices found by Meixner \cite{meixner}.
Likewise,
the Krein array
$\{56,45,16,1;1,8,45,56 \}$
is formally dual to the intersection array
of an antipodal distance-regular triple cover
found by Soicher \cite{soicher} which is not cometric.

\section{Five-class cometric Q-antipodal association
schemes}\label{sec:five}

In \cite{Hsrd2}, Higman introduced so-called strongly regular designs of the
second kind and showed that these are equivalent to coherent configurations of
type [3 3;~ 3]. Because such a coherent configuration is
balanced --- a concept defined by Hirasaka and Sharafdini \cite{HR} --- its two
fibres necessarily have the same size (this was also observed by Higman \cite{Hsrd2}). If in
addition the design has self-dual parameters, then it gives rise to a
five-class uniform scheme.

A trivial way to obtain such a scheme is by taking the bipartite double of a
strongly regular graph (Higman calls the corresponding strongly regular design
of the second kind trivial). Though trivial, there are some cometric (and also
metric) schemes obtained in this way, such as the ones obtained from the
Clebsch graph, Schl\"{a}fli graph, Higman-Sims graph, the McLaughlin graph and
both its subconstituents. These strongly regular graphs have in common that
$q^1_{11}=0$ and $q^2_{12} \neq 0$. It was in fact claimed by Bannai and Ito
\cite[p.\ 314]{banito} that the bipartite double of a scheme is cometric if and
only if the (original) scheme is cometric with $q^i_{1i}=0$ for $i \neq d$ and
$q^d_{id} \neq 0$.

To obtain less trivial examples of cometric schemes, we checked
the examples and table of parameter sets for nontrivial
strongly regular designs of the second kind in \cite{Hsrd2}.
Four parameter sets in the table there turn out to give
cometric schemes. One with $n=162$ (Higman's Example 4.4) is
related to $U_4(3)$, and has Krein array
$\{21,20,9,3,1;1,3,9,20,21\}$. The second one (Higman's Example
4.5) has $n=176$, and can be described using the Steiner
3-design on 22 points. It has Krein array
$\{21,19.36,11,2.64,1;1,2.64,11,19.36,21\}$. The parameter set
with $n=243$ can be realized as a dismantled scheme on two of
the three fibres of a cometric scheme that is the dual of a
metric scheme corresponding to the coset graph of the shortened
extended ternary Golay code (cf. \cite[p.\ 365]{bcn}). Its Krein
array is $\{22,20,13.5,2,1;1,2,13.5,20,22\}$. The last cometric
example from the table has $n=256$ (second such parameter set
in Higman's table) and corresponds to the distance-regular
folded 10-cube.

Higman also mentions (in his Example 4.3) the strongly regular
designs of the second kind related to the family of bipartite
cometric distance-regular dual polar graphs $D_5(q)$. We did
not bother to completely check all other examples mentioned by
Higman \cite{Hsrd2}, but we expect no other cometric examples
among these.

\section{Miscellaneous}
\label{Sec:misc}

In his book on permutation groups, Cameron \cite[p.\ 79]{cameronbook}
describes how to use the computer package GAP to construct the
strongly regular decomposition of the Higman-Sims graph into two
Hoffman-Singleton graphs. This description can easily be extended to
get the linked system of
partial $\lambda$-geometries of Section \ref{HOSI}.

We checked whether any of the remaining examples
mentioned in Higman's unpublished paper on uniform schemes
\cite{Huninform} gives rise to a cometric scheme. Although we
should mention that one of the examples (Example 6) is unclear
to us, we found no cometric schemes among these examples.

Many of the examples mentioned in this paper, and also examples of other cometric association schemes, are listed on
the website \cite{Billwebsite}. Included there are all parameters of the examples.\\

\noindent {\bf Acknowledgements} The authors thank Peter Cameron, Bill Kantor, and Tim Penttila for inspiring
discussions on the topic of this paper, and the referees for comments and corrections on an earlier version.

\bibliographystyle{amsplain}

\newpage
\section*{Appendix}

Below are putative parameter sets of four-class cometric Q-antipodal
association schemes with fibre size $n \leq 2000$ and $w \leq 6$,
and that are not Q-bipartite. The parameter sets are grouped according
to the parameters of the strongly regular graph which would appear
as the subscheme on the fibres. These ``srg'' parameters are given at
the beginning of each group. An exclamation mark (!) means that the strongly regular graph is
unique and a plus sign (+) indicates existence. Most of this information is obtained from
the online tables of strongly regular graphs by Brouwer \cite{Brouwersrgtables}.
For each group of parameter sets, we give the absolute bound of
Proposition \ref{absolutebound} (if relevant).

Each remaining line contains information on one
parameter set. Again, an exclamation mark (!) means that the scheme is
unique, a plus sign (+) indicates existence, and a minus sign (-)
non-existence. Next to this, the Krein array is given, then $w$,
the partition $v =1+  v_1+ v_2+  v_3+ v_4$, and the
spectrum of $R_1$. At the end, some miscellaneous information is
given. The notation {\tt (P)} indicates that the scheme is (or, would be)
also metric. The examples listed here appear in the body of the paper as follows.

\begin{itemize}
\item {\tt Hoff-Singleton} -- the linked system of partial $\lambda$-geometries
related to the Hoffman-Singleton graph (Section \ref{HOSI})
\item {\tt hemisystem} -- schemes arising from hemisystems (Corollary \ref{Chemisystems})
\item {\tt ovals of PG(2,4)} -- Higman's scheme defined on the ovals of $PG(2,4)$
(Section \ref{smallcases})
\item {\tt folded 8-cube} -- (Section \ref{smallcases})
\item {\tt ternary Golay code} -- the decomposition of Goethals and Seidel in \cite{GS}
(Section \ref{smallcases})
\item {\tt D\_4(q)} and  {\tt O+(8,q), triality} -- Higman's triality schemes and their dismantled schemes (Example
    \ref{Ex-triality})
\item {\tt Leech lattice} -- Higman's Leech lattice example
    \cite[Ex. 4]{Htriality} (Section \ref{smallcases})\\
\end{itemize}

{\tiny
\begin{verbatim}


--------------------
   !srg(50,42,35,36)                          w <= 7
!{21, 16, 6,  1; 1, 6,  16, 21}               2  100=1+  15+ 42+  35+  7    15   5 0  -5  -15 Hoff-Singleton (P)
!{21, 16, 8,  1; 1, 4,  16, 21}               3  150=1+  30+ 42+  70+  7    30  10 0  -5  -15 Hoff-Singleton
-{21, 16, 9,  1; 1, 3,  16, 21}               4  200=1+  45+ 42+ 105+  7    45  15 0  -5  -15
-{21, 16, 9.6,1; 1, 2.4,16, 21}               5  250=1+  60+ 42+ 140+  7    60  20 0  -5  -15
-{21, 16, 10, 1; 1, 2,  16, 21}               6  300=1+  75+ 42+ 175+  7    75  25 0  -5  -15
--------------------
   !srg(56,45,36,36)                          w <= 5
!{20, 16.333, 4.667, 1; 1, 4.667, 16.333, 20} 2  112=1+  20+ 45+  36+ 10    20   6 0  -6  -20 hemisystem
!{20, 16.333, 6.222, 1; 1, 3.111, 16.333, 20} 3  168=1+  40+ 45+  72+ 10    40  12 0  -6  -20 ovals of PG(2,4)
-{20, 16.333, 7,     1; 1, 2.333, 16.333, 20} 4  224=1+  60+ 45+ 108+ 10    60  18 0  -6  -20
-{20, 16.333, 7.467, 1; 1, 1.867, 16.333, 20} 5  280=1+  80+ 45+ 144+ 10    80  24 0  -6  -20
--------------------
   +srg(64,28,12,12)
!{28, 15, 6, 1; 1, 6, 15, 28}                 2  128=1+   8+ 28+  56+ 35     8   4 0  -4   -8 folded 8-cube (P)
--------------------
   !srg(81,60,45,42)                          w <= 3
!{20, 18, 3,  1; 1, 3,  18, 20}               2  162=1+  36+ 60+  45+ 20    36   9 0  -9  -36 ternary Golay code
+{20, 18, 4,  1; 1, 2,  18, 20}               3  243=1+  72+ 60+  90+ 20    72  18 0  -9  -36 ternary Golay code
---------------------
   +srg(135,70,37,35)                         w <= 14
+{50, 31.5,  9.375, 1; 1, 9.375, 31.5, 50}    2  270=1+  15+ 70+ 120+ 64    15   6 0  -6  -15 D_4(2) (P)
+{50, 31.5, 12.5,   1; 1, 6.25,  31.5, 50}    3  405=1+  30+ 70+ 240+ 64    30  12 0  -6  -15 O+(8,2), triality
 {50, 31.5, 14.0625,1; 1, 4.6875,31.5, 50}    4  540=1+  45+ 70+ 360+ 64    45  18 0  -6  -15
 {50, 31.5, 15,     1; 1, 3.75,  31.5, 50}    5  675=1+  60+ 70+ 480+ 64    60  24 0  -6  -15
 {50, 31.5, 15.625, 1; 1, 3.125, 31.5, 50}    6  810=1+  75+ 70+ 600+ 64    75  30 0  -6  -15
------------------------


    srg(162,140,121,120)                      w <= 14
 {56, 45, 12,  1; 1, 12,  45, 56}             2  324=1+  36+140+ 126+ 21    36   9 0  -9  -36 (P)
 {56, 45, 16,  1; 1,  8,  45, 56}             3  486=1+  72+140+ 252+ 21    72  18 0  -9  -36
 {56, 45, 18,  1; 1,  6,  45, 56}             4  648=1+ 108+140+ 378+ 21   108  27 0  -9  -36
 {56, 45, 19.2,1; 1,  4.8,45, 56}             5  810=1+ 144+140+ 504+ 21   144  36 0  -9  -36
 {56, 45, 20,  1; 1,  4,  45, 56}             6  972=1+ 180+140+ 630+ 21   180  45 0  -9  -36
------------------------
    srg(196,150,116,110)                      w <= 6
 {45, 40,  6,  1; 1, 6,  40, 45}              2  392=1+  70+150+ 126+ 45     70  14 0 -14  -70
 {45, 40,  8,  1; 1, 4,  40, 45}              3  588=1+ 140+150+ 252+ 45    140  28 0 -14  -70
 {45, 40,  9,  1; 1, 3,  40, 45}              4  784=1+ 210+150+ 378+ 45    210  42 0 -14  -70
 {45, 40,  9.6,1; 1, 2.4,40, 45}              5  980=1+ 280+150+ 504+ 45    280  56 0 -14  -70
 {45, 40, 10,  1; 1, 2,  40, 45}              6 1176=1+ 350+150+ 630+ 45    350  70 0 -14  -70
------------------------
    srg(243,176,130,120)                      w <= 4
 {44, 40.5, 4.5, 1; 1, 4.5, 40.5, 44}         2  486=1+  99+176+ 144+ 66     99  18 0 -18  -99
 {44, 40.5, 6,   1; 1, 3,   40.5, 44}         3  729=1+ 198+176+ 288+ 66    198  36 0 -18  -99
 {44, 40.5, 6.75,1; 1, 2.25,40.5, 44}         4  972=1+ 297+176+ 432+ 66    297  54 0 -18  -99
------------------------
    srg(320,220,156,140)                      w <= 3
 {44, 41.667, 3.333, 1; 1, 3.333, 41.667, 44} 2  640=1+ 144+220+ 176+ 99    144  24 0 -24 -144
 {44, 41.667, 4.444, 1; 1, 2.222, 41.667, 44} 3  960=1+ 288+220+ 352+ 99    288  48 0 -24 -144
------------------------
   +srg(378,325,280,275)                      w <= 19
+{104, 88.2, 16.8, 1; 1, 16.8,88.2, 104}      2  756=1+  78+325+ 300+ 52     78  15 0 -15  -78 hemisystem
 {104, 88.2, 22.4, 1; 1, 11.2,88.2, 104}      3 1134=1+ 156+325+ 600+ 52    156  30 0 -15  -78
 {104, 88.2, 25.2, 1; 1, 8.4, 88.2, 104}      4 1512=1+ 234+325+ 900+ 52    234  45 0 -15  -78
 {104, 88.2, 26.88,1; 1, 6.72,88.2, 104}      5 1890=1+ 312+325+1200+ 52    312  60 0 -15  -78
 {104, 88.2, 28,   1; 1, 5.6, 88.2, 104}      6 2268=1+ 390+325+1500+ 52    390  75 0 -15  -78
------------------------
    srg(392,345,304,300)                      w <= 23
 {115, 96, 20,    1; 1, 20,    96, 115}       2  784=1+  70+345+ 322+ 46     70  14 0 -14  -70 (P)
 {115, 96, 26.667,1; 1, 13.333,96, 115}       3 1176=1+ 140+345+ 644+ 46    140  28 0 -14  -70
 {115, 96, 30,    1; 1, 10,    96, 115}       4 1568=1+ 210+345+ 966+ 46    210  42 0 -14  -70
 {115, 96, 32,    1; 1,  8,    96, 115}       5 1960=1+ 280+345+1288+ 46    280  56 0 -14  -70
 {115, 96, 33.333,1; 1,  6.667,96, 115}       6 2352=1+ 350+345+1610+ 46    350  70 0 -14  -70
------------------------
    srg(400,315,250,240)                      w <= 11
 {84, 75, 10,    1; 1, 10,   75, 84}          2  800=1+ 120+315+ 280+ 84    120  20 0 -20 -120
 {84, 75, 13.333,1; 1, 6.667,75, 84}          3 1200=1+ 240+315+ 560+ 84    240  40 0 -20 -120
 {84, 75, 15,    1; 1, 5,    75, 84}          4 1600=1+ 360+315+ 840+ 84    360  60 0 -20 -120
 {84, 75, 16,    1; 1, 4,    75, 84}          5 2000=1+ 480+315+1120+ 84    480  80 0 -20 -120
 {84, 75, 16.667,1; 1, 3.333,75, 84}          6 2400=1+ 600+315+1400+ 84    600 100 0 -20 -120
------------------------
    srg(540,385,280,260)                      w <= 6
 {77, 72,  6,  1; 1, 6,  72, 77}              2 1080=1+ 210+385+ 330+154    210  30 0 -30 -210
 {77, 72,  8,  1; 1, 4,  72, 77}              3 1620=1+ 420+385+ 660+154    420  60 0 -30 -210
 {77, 72,  9,  1; 1, 3,  72, 77}              4 2160=1+ 630+385+ 990+154    630  90 0 -30 -210
 {77, 72,  9.6,1; 1, 2.4,72, 77}              5 2700=1+ 840+385+1320+154    840 120 0 -30 -210
 {77, 72, 10,  1; 1, 2,  72, 77}              6 3240=1+1050+385+1650+154   1050 150 0 -30 -210
------------------------
    srg(672,440,292,280)
 {176, 135, 24, 1; 1, 24, 135, 176}           2 1344=1+  56+440+ 616+231     56  14 0 -14  -56 (P)
------------------------
    srg(704,475,330,300)                      w <= 4
 {76, 72.6, 4.4,  1; 1, 4.4,  72.6, 76}       2 1408=1+ 304+475+ 400+228    304  40 0 -40 -304
 {76, 72.6, 5.867,1, 1, 2.933,72.6, 76}       3 2112=1+ 608+475+ 800+228    608  80 0 -40 -304
 {76, 72.6, 6.6,  1; 1, 2.2,  72.6, 76}       4 2816=1+ 912+475+1200+228    912 120 0 -40 -304
------------------------
    srg(729,588,477,462)                      w <= 16
 {140, 126, 15,  1; 1,15,  126, 140}          2 1458=1+ 189+ 588+ 540+140   189  27 0 -27 -189
 {140, 126, 20,  1; 1,10,  126, 140}          3 2187=1+ 378+ 588+1080+140   378  54 0 -27 -189
 {140, 126, 22.5,1; 1, 7.5,126, 140}          4 2916=1+ 567+ 588+1620+140   567  81 0 -27 -189
 {140, 126, 24,  1; 1, 6,  126, 140}          5 3645=1+ 756+ 588+2160+140   756 108 0 -27 -189
 {140, 126, 25,  1; 1, 5,  126, 140}          6 4374=1+ 945+ 588+2700+140   945 135 0 -27 -189
------------------------
    srg(760,594,468,450)                      w <= 13
 {132, 120.333,12.667,1;1,12.667,120.333,132} 2 1520=1+ 220+ 594+ 540+165   220  30 0 -30 -220
 {132, 120.333,16.889,1;1, 8.444,120.333,132} 3 2280=1+ 440+ 594+1080+165   440  60 0 -30 -220
 {132, 120.333,19,    1;1, 6.333,120.333,132} 4 3040=1+ 660+ 594+1620+165   660  90 0 -30 -220
 {132, 120.333,20.267,1;1, 5.067,120.333,132} 5 3800=1+ 880+ 594+2160+165   880 120 0 -30 -220
 {132, 120.333,21.111,1;1, 4.222,120.333,132} 6 4560=1+1100+ 594+2700+165  1100 150 0 -30 -220
------------------------
    srg(800,714,638,630)                      w <= 34
 {204, 175, 30, 1; 1, 30, 175, 204}           2 1600=1+ 120+ 714+ 680+ 85   120  20 0 -20 -120 (P)
 {204, 175, 40, 1; 1, 20, 175, 204}           3 2400=1+ 240+ 714+1360+ 85   240  40 0 -20 -120
 {204, 175, 45, 1; 1, 15, 175, 204}           4 3200=1+ 360+ 714+2040+ 85   360  60 0 -20 -120
 {204, 175, 48, 1; 1, 12, 175, 204}           5 4000=1+ 480+ 714+2720+ 85   480  80 0 -20 -120
 {204, 175, 50, 1; 1, 10, 175, 204}           6 4800=1+ 600+ 714+3400+ 85   600 100 0 -20 -120
------------------------
    srg(875,570,385,345)                      w <= 3
 {76, 73.5, 3.5,   1; 1, 3.5,  73.5, 76}      2 1750=1+ 400+ 570+ 475+304   400  50 0 -50 -400
 {76, 73.5, 4.667, 1; 1, 2.333,73.5, 76}      3 2625=1+ 800+ 570+ 950+304   800 100 0 -50 -400
-------------------------
   +srg(1120,390,146,130)                     w <= 54
+{300, 212.333,38.889,1;1,38.889,212.333,300} 2 2240=1+  40+ 390+1080+729    40  12 0 -12  -40 D_4(3) (P)
+{300, 212.333,51.852,1;1,25.926,212.333,300} 3 3360=1+  80+ 390+2160+729    80  24 0 -12  -40 O+(8,3), triality
 {300, 212.333,58.333,1;1,19.444,212.333,300} 4 4480=1+ 120+ 390+3240+729   120  36 0 -12  -40
 {300, 212.333,62.222,1;1,15.556,212.333,300} 5 5600=1+ 160+ 390+4320+729   160  48 0 -12  -40
 {300, 212.333,64.815,1;1,12.963,212.333,300} 6 6720=1+ 200+ 390+5400+729   200  60 0 -12  -40
-------------------------
    srg(1210,819,568,525)                     w <= 6
 {117, 112, 6,  1; 1, 6,  112, 117}           2 2420=1+ 495+ 819+ 715+390   495  55 0 -55 -495
 {117, 112, 8,  1; 1, 4,  112, 117}           3 3630=1+ 990+ 819+1430+390   990 110 0 -55 -495
 {117, 112, 9,  1; 1, 3,  112, 117}           4 4840=1+1485+ 819+2145+390  1485 165 0 -55 -495
 {117, 112, 9.6,1; 1, 2.4,112, 117}           5 6050=1+1980+ 819+2860+390  1980 220 0 -55 -495
 {117, 112,10,  1; 1, 2,  112, 117}           6 7260=1+2475+ 819+3575+390  2475 275 0 -55 -495
--------------------------
    srg(1225,1008,833,812)                    w <= 23
 {216, 196, 21,  1; 1,21,  196, 216}          2 2450=1+ 280+1008+ 945+216   280  35 0 -35 -280
 {216, 196, 28,  1; 1,14,  196, 216}          3 3675=1+ 560+1008+1890+216   560  70 0 -35 -280
 {216, 196, 31.5,1; 1,10.5,196, 216}          4 4900=1+ 840+1008+2835+216   840 105 0 -35 -280
 {216, 196, 33.6,1; 1, 8.4,196, 216}          5 6125=1+1120+1008+3780+216  1120 140 0 -35 -280
 {216, 196, 35,  1; 1, 7,  196, 216}          6 7350=1+1400+1008+4725+216  1400 175 0 -35 -280
----------------------------
   +srg(1376,1225,1092,1078)                  w <= 41
+{300, 264.143,36.857,1;1,36.857,264.143,300} 2 2752=1+ 200+1225+1176+150   200  28 0 -28 -200 hemisystem
 {300, 264.143,49.143,1;1,24.571,264.143,300} 3 4128=1+ 400+1225+2352+150   400  56 0 -28 -200
 {300, 264.143,55.286,1;1,18.429,264.143,300} 4 5504=1+ 600+1225+3528+150   600  84 0 -28 -200
 {300, 264.143,58.971,1;1,14.743,264.143,300} 5 6880=1+ 800+1225+4704+150   800 112 0 -28 -200
 {300, 264.143,61.429,1;1,12.286,264.143,300} 6 8256=1+1000+1225+5880+150  1000 140 0 -28 -200
-------------------------
   +srg(1408,567,246,216)                     w <= 27
+{252, 201.667,22,    1;1,22,    201.667,252} 2 2816=1+ 112+ 567+1296+840   112  24 0 -24 -112 Leech lattice
+{252, 201.667,29.333,1;1,14.667,201.667,252} 3 4224=1+ 224+ 567+2592+840   224  48 0 -24 -112 Leech lattice
 {252, 201.667,33,    1;1,11,    201.667,252} 4 5632=1+ 336+ 567+3888+840   336  72 0 -24 -112
 {252, 201.667,35.2,  1;1, 8.8,  201.667,252} 5 7040=1+ 448+ 567+5184+840   448  96 0 -24 -112
 {252, 201.667,36.667,1;1, 7.333,201.667,252} 6 8448=1+ 560+ 567+6480+840   560 120 0 -24 -112
----------------------------
    srg(1458,1316,1189,1176)                  w <= 47
 {329, 288, 42,  1; 1, 42,  288, 329}         2 2916=1+ 189+1316+1269+141   189  27 0 -27 -189 (P)
 {329, 288, 56,  1; 1, 28,  288, 329}         3 4374=1+ 378+1316+2538+141   378  54 0 -27 -189
 {329, 288, 63,  1; 1, 21,  288, 329}         4 5832=1+ 567+1316+3807+141   567  81 0 -27 -189
 {329, 288, 67.2,1; 1, 16.8,288, 329}         5 7290=1+ 756+1316+5076+141   756 108 0 -27 -189
 {329, 288, 70,  1; 1, 14,  288, 329}         6 8748=1+ 945+1316+6345+141   945 135 0 -27 -189
--------------------------
    srg(1625,1044,693,630)                    w <= 4
 {116, 112.667, 4.333,1;1,4.333,112.667, 116} 2 3250=1+ 725+1044+ 900+580   725  75 0 -75 -725
 {116, 112.667, 5.778,1;1,2.889,112.667, 116} 3 4875=1+1450+1044+1800+580  1450 150 0 -75 -725
 {116, 112.667, 6.5,  1;1,2.167,112.667, 116} 4 6500=1+2175+1044+2700+580  2175 225 0 -75 -725
--------------------------
    srg(1701,1190,847,798)                    w <= 9
 {170, 162,  9,  1; 1, 9,  162, 170}          2 3402=1+ 630+1190+1071+510   630  63 0 -63 -630
 {170, 162, 12,  1; 1, 6,  162, 170}          3 5103=1+1260+1190+2142+510  1260 126 0 -63 -630
 {170, 162, 13.5,1; 1, 4.5,162, 170}          4 6804=1+1890+1190+3213+510  1890 189 0 -63 -630
 {170, 162, 14.4,1; 1, 3.6,162, 170}          5 8505=1+2520+1190+4284+510  2520 252 0 -63 -630
 {170, 162, 15,  1; 1, 3,  162, 170}         6 10206=1+3150+1190+5355+510  3150 315 0 -63 -630
----------------------------
    srg(1936,1620,1360,1332)                  w <= 30
 {315, 288, 28,    1; 1, 28,    288, 315}     2 3872=1+ 396+1620+1540+315   396  44 0 -44 -396
 {315, 288, 37.333,1; 1, 18.667,288, 315}     3 5808=1+ 792+1620+3080+315   792  88 0 -44 -396
 {315, 288, 42,    1; 1, 14,    288, 315}     4 7744=1+1188+1620+4620+315  1188 132 0 -44 -396
 {315, 288, 44.8,  1; 1, 11.2,  288, 315}     5 9680=1+1584+1620+6160+315  1584 176 0 -44 -396
 {315, 288, 46.667,1; 1,  9.333,288, 315}    6 11616=1+1980+1620+7700+315  1980 220 0 -44 -396
--------------------------
    srg(1944,1218,792,714)                    w <= 3
 {116, 113.4, 3.6, 1; 1, 3.6, 113.4, 116}     2 3888=1+ 900+1218+1044+725   900  90 0 -90 -900
 {116, 113.4, 4.8, 1; 1, 2.4, 113.4, 116}     3 5832=1+1800+1218+2088+725  1800 180 0 -90 -900
--------------------------
\end{verbatim}
}

\end{document}